\documentclass[11pt,a4paper,reqno]{amsart}
\usepackage[T1]{fontenc}
\usepackage{lmodern}
\usepackage[a4paper,margin=28mm]{geometry}
\usepackage{amsmath,amssymb,amsthm,mathtools}
\usepackage[mathscr]{euscript}
\usepackage{microtype}
\usepackage{enumitem}
\usepackage{needspace}
\usepackage{url}
\usepackage{xcolor}
\usepackage[colorlinks=true,linkcolor=blue,citecolor=blue,urlcolor=blue]{hyperref}
\usepackage{aliascnt}
\usepackage[nameinlink,noabbrev]{cleveref}

\newtheorem{maintheorem}{Theorem}

\crefname{maintheorem}{theorem}{theorems}
\Crefname{maintheorem}{Theorem}{Theorems}
\newaliascnt{maincorollary}{maintheorem}
\newtheorem{maincorollary}[maincorollary]{Corollary}
\aliascntresetthe{maincorollary}
\crefname{maincorollary}{corollary}{corollaries}
\Crefname{maincorollary}{Corollary}{Corollaries}
\newtheorem{theorem}{Theorem}[section]
\newaliascnt{lemma}{theorem}
\newtheorem{lemma}[lemma]{Lemma}
\aliascntresetthe{lemma}
\crefname{lemma}{lemma}{lemmas}
\Crefname{lemma}{Lemma}{Lemmas}
\newaliascnt{proposition}{theorem}
\newtheorem{proposition}[proposition]{Proposition}
\aliascntresetthe{proposition}
\crefname{proposition}{proposition}{propositions}
\Crefname{proposition}{Proposition}{Propositions}
\newaliascnt{corollary}{theorem}
\newtheorem{corollary}[corollary]{Corollary}
\aliascntresetthe{corollary}
\crefname{corollary}{corollary}{corollaries}
\Crefname{corollary}{Corollary}{Corollaries}
\theoremstyle{definition}
\newaliascnt{remark}{theorem}
\newtheorem{remark}[remark]{Remark}
\aliascntresetthe{remark}
\crefname{remark}{remark}{remarks}
\Crefname{remark}{Remark}{Remarks}

\newlist{resultparts}{enumerate}{1}
\setlist[resultparts]{label=\textup{(\roman*)},ref=(\roman*),
leftmargin=2em,topsep=3pt,itemsep=2pt,parsep=0pt}

\newenvironment{AIstatement}
  {\par\medskip\noindent\textbf{AI statement.}\ \ignorespaces}
  {\par}
\newenvironment{acknowledgements}
  {\par\medskip\noindent\textbf{Acknowledgements.}\ \ignorespaces}
  {\par}

\newcommand{\N}{\mathbb N}
\newcommand{\R}{\mathbb R}
\newcommand{\C}{\mathbb C}
\newcommand{\T}{\mathbb T}
\newcommand{\eps}{\varepsilon}
\newcommand{\norm}[1]{\lVert #1\rVert}

\DeclareMathOperator{\tr}{tr}
\DeclareMathOperator{\rank}{rank}
\DeclareMathOperator{\diag}{diag}
\DeclareMathOperator{\ran}{ran}

\numberwithin{equation}{section}
\allowdisplaybreaks[1]
\hypersetup{pdftitle={A negative solution to the complemented subspace problem for Banach spaces with unconditional bases},
pdfauthor={Antonio Acuaviva}, pdfsubject={Circle tensor spaces and complemented subspaces of Banach lattices}, pdfkeywords={unconditional basis, complemented subspace, Banach lattice, local unconditional structure, superreflexive space}}

\title[The complemented subspace problem]{A negative solution to the complemented subspace problem for Banach spaces with unconditional bases}
\author[A.~Acuaviva]{Antonio Acuaviva}
\address{(A.~Acuaviva) School of Mathematical Sciences, Charles Carter Building,
Lancaster University, Lancaster LA1 4YX, United Kingdom}
\email{ahacua@gmail.com}
\date{5 September 2026}

\begin{document}
\begin{abstract}
We give a negative solution to the complemented subspace problem for Banach spaces with unconditional bases over both the real and complex fields. For every $\rho>0$, we construct a projection $P_\rho$ of norm less than $1+\rho$ on a separable superreflexive space
\begin{equation*}
 X_\rho=\left(\bigoplus_{j=1}^{\infty}\ell_{p_j}^{N_j}\right)_2,
 \qquad p_j\downarrow2,
\end{equation*}
such that $Z_\rho=P_\rho(X_\rho)$ and its dual $Z_\rho^*$ have Schauder bases but admit no unconditional bases. Over the real field, both spaces have Gordon--Lewis local unconditional structure (GL-lust) but fail Dubinsky--Pe\l czy\'nski--Rosenthal local unconditional structure (DPR-lust), disproving a conjecture of Figiel, Johnson and Tzafriri. In particular, neither is isomorphic to a Banach lattice, giving a negative solution to the separable Banach-lattice complemented subspace problem. A modification of the construction also shows that the class of separable real Banach lattices is not primary. A Lean~4 formalisation of the main results accompanies the paper.
\end{abstract}
\maketitle

\begin{center}
\emph{Companion Lean~4 formalisation:}
\url{https://banach-complemented-subspaces.github.io/}
\end{center}

\tableofcontents

\section{Introduction}\label{sec:main}

The \emph{complemented subspace problem} (CSP) asks whether every complemented subspace of a space in a given class is isomorphic to a space in the same class. Classical instances include Pe\l czy\'nski's questions for $L_1$-spaces and $C(K)$-spaces, posed in 1960 \cite[p.~210]{Pelczynski1960}.

The question for spaces with an unconditional basis, traditionally attributed to Banach \cite{DeHeviaEtAl2025}, is also recorded as Problem~1.d.4 in Lindenstrauss and Tzafriri's monograph \cite[p.~27]{LindenstraussTzafriri1977}. One approach is to study whether an unconditional basis of a direct sum can be partitioned into subsequences spanning copies of its summands. Such splitting results were obtained, under additional hypotheses, by Edelstein and Wojtaszczyk \cite{EdelsteinWojtaszczyk1976} and later by Wojtaszczyk \cite{Wojtaszczyk1978}.

The corresponding problem for Banach lattices is closely related: both $L_1$-spaces and $C(K)$-spaces are Banach lattices, and a real space with an unconditional basis can be equivalently renormed as a Banach lattice under the coordinate order.

Plebanek and Salguero-Alarc\'on \cite{PlebanekSalguero2023} gave a negative answer to Pe\l czy\'nski's $C(K)$-space question. They constructed a $1$-complemented subspace $\mathrm{PS}_2$ of a $C(K)$-space which is not isomorphic to any $C(K)$-space. De Hevia, Mart\'inez-Cervantes, Salguero-Alarc\'on and Tradacete \cite{DeHeviaEtAl2025} subsequently proved that $\mathrm{PS}_2$ is not isomorphic to any Banach lattice, giving a negative solution to the Banach-lattice problem as well.

The space $\mathrm{PS}_2$ is nonseparable, so these results do not settle the problem for separable Banach lattices. For a fuller account of the CSP and its history, we refer to the survey of de Hevia and Tradacete~\cite{DeHeviaTradacete2025}.

We give a negative solution to the CSP for unconditional bases over both the real and complex fields, and for separable real Banach lattices. The examples are complemented subspaces of superreflexive spaces with a $1$-unconditional basis, and the projections have norms arbitrarily close to one. Closely related mixed sums of finite-dimensional $\ell_p$-spaces were used by Figiel in his classical construction of a reflexive Banach space not isomorphic to its Cartesian square \cite{Figiel1972}. Our conclusions follow from separating two notions of local unconditional structure.

Roughly, the Dubinsky--Pe\l czy\'nski--Rosenthal version (DPR-lust) \cite{DubinskyPelczynskiRosenthal1972} asks that each finite-dimensional subspace lie in a larger finite-dimensional subspace with a uniformly controlled unconditional basis. The Gordon--Lewis version (GL-lust) \cite{GordonLewis1974} instead uses uniformly bounded factorisations through spaces with unconditional bases, which need not be subspaces of the given space. Thus DPR-lust implies GL-lust, but whether the converse holds was unknown \cite{DeHeviaTradacete2025}; Figiel, Johnson and Tzafriri already emphasised the distinction between the two notions \cite[Remark~2.3]{FigielJohnsonTzafriri1975}.

Every Banach lattice has DPR-lust, while every complemented subspace of a Banach lattice has GL-lust. The main conjecture of Figiel, Johnson and Tzafriri \cite[p.~396]{FigielJohnsonTzafriri1975} states that such complemented subspaces also have DPR-lust. The following theorem disproves this conjecture and, in particular, separates GL-lust from DPR-lust in separable superreflexive spaces.

\begin{samepage}
\begin{maintheorem}\label{thm:real-consequences}
For every $\rho>0$, there exist sequences $(N_j)_{j=1}^{\infty}$ in $\N$ and $(p_j)_{j=1}^{\infty}$ in $(2,3]$, with $p_j\downarrow2$, and a projection $P_\rho$ on the separable superreflexive real Banach lattice
\begin{equation*}
 X_\rho=\left(\bigoplus_{j=1}^{\infty}\ell_{p_j}^{N_j}(\R)\right)_2
\end{equation*}
such that, with $Z_\rho=P_\rho(X_\rho)$, we have:
\begin{resultparts}
\item\label{it:main-separation}
For $V\in\{Z_\rho,Z_\rho^*\}$,
\begin{equation*}
 \chi_{\mathrm{GL}}(V)\leq\norm{P_\rho}<1+\rho,
 \qquad \chi_{\mathrm{DPR}}(V)=\infty.
\end{equation*}
\item\label{it:main-complement}
The projection $P_\rho$ can be chosen so that $\norm{I_{X_\rho}-P_\rho}<1+\rho$ and \ref{it:main-separation} also holds for $(I_{X_\rho}-P_\rho)(X_\rho)$ and its dual, with $P_\rho$ replaced by $I_{X_\rho}-P_\rho$.
\end{resultparts}
\end{maintheorem}
\end{samepage}

It follows that DPR-lust is not inherited by complemented subspaces. Since every Banach lattice has DPR-lust and this property is invariant under isomorphism, neither $Z_\rho$ nor $Z_\rho^*$ is isomorphic to a Banach lattice. In particular, neither space admits an unconditional basis. Together with the complex construction, this gives the following negative solutions to the CSP for unconditional bases and for separable real Banach lattices.

\Needspace{15\baselineskip}
\begin{samepage}
\begin{maintheorem}\label{thm:main}
For $\mathbb K\in\{\R,\C\}$ and every $\rho>0$, there are a Banach space $X_\rho$ over $\mathbb K$ with a $1$-unconditional basis and a projection $P_\rho$ on $X_\rho$ with $\norm{P_\rho}<1+\rho$ such that, writing $Z_\rho=P_\rho(X_\rho)$, neither $Z_\rho$ nor $Z_\rho^*$ admits an unconditional basis. If $\mathbb K=\R$, neither space is isomorphic to a Banach lattice.
\end{maintheorem}
\end{samepage}

The spaces $Z_\rho$ and $Z_\rho^*$ are superreflexive and admit shrinking, boundedly complete bases, by \Cref{cor:schauder-bases}\ref{it:schauder-existence}. By \Cref{cor:schauder-bases}\ref{it:schauder-tails}, these bases can be chosen to be asymptotically monotone.

The projection bound is sharp over $\C$. Kalton and Wood \cite{KaltonWood1976} showed that every contractively complemented subspace of a complex Banach space with a $1$-unconditional basis again admits a $1$-unconditional basis; see also \cite[Theorem~7.18]{Randrianantoanina2001}. Thus the complex example satisfies $1<\norm{P_\rho}<1+\rho$, with $\rho$ arbitrarily small. The corresponding $1$-unconditional conclusion fails over $\R$: counterexamples were obtained by Lewis \cite{Lewis1979} and by Benyamini, Flinn and Lewis \cite{BenyaminiFlinnLewis1984}; see also \cite[Proposition~7.25]{Randrianantoanina2001}.

De Hevia and Tradacete also asked whether, in a decomposition of a Banach lattice into two closed subspaces, at least one must be isomorphic to a Banach lattice \cite[Question~4]{DeHeviaTradacete2025}. This is the \emph{primariness} question for the class; see \Cref{subsec:two-complementary-ranges}.

In earlier work \cite{Acuaviva2026}, we gave a negative answer by decomposing a nonseparable $C(K)$-space into two subspaces neither isomorphic to a Banach lattice. \Cref{thm:real-consequences}\ref{it:main-complement} gives the following separable counterpart.

\Needspace{5\baselineskip}
\begin{maincorollary}\label{cor:separable-nonprimarity}
The class of separable real Banach lattices is not primary.
\end{maincorollary}

The earlier decomposition also shows that the class of $C(K)$-spaces is not primary \cite{Acuaviva2026}. In the present setting, \Cref{thm:real-consequences}\ref{it:main-complement} gives the same conclusion for separable real spaces with DPR-lust and for those with an unconditional basis.

All four spaces in \Cref{thm:real-consequences}\ref{it:main-complement} also have a contractive unconditional finite-dimensional decomposition and the metric approximation property, as shown in \Cref{cor:two-nonlattice-ranges}\ref{it:two-fdd}. The Kalton--Peck space has an unconditional decomposition into two-dimensional subspaces \cite{KaltonPeck1979}, while Johnson, Lindenstrauss and Schechtman \cite{JohnsonLindenstraussSchechtman1980} showed that it fails GL-lust. Casazza and Kalton \cite[Theorem~3.8]{CasazzaKalton1996} proved that GL-lust together with an unconditional finite-dimensional decomposition of uniformly bounded block dimensions forces an unconditional basis. Thus the unconditional finite-dimensional decompositions in the present examples necessarily have unbounded block dimensions.

Over $\R$, Lacey's stabilisation theorem \cite[p.~49, second Corollary]{Lacey1977} states that a Banach space $V$ has GL-lust if and only if $V\oplus c_0$ has DPR-lust. Consequently, $Z_\rho\oplus_\infty c_0$ has DPR-lust, although its contractively complemented subspace $Z_\rho$ does not. Thus DPR-lust fails to pass even to contractively complemented subspaces.

\Needspace{8\baselineskip}
\subsection*{Idea of the proof and organisation}

\Cref{sec:preliminaries} introduces the notation and recalls the two notions of local unconditional structure, together with the Hilbert-space and probability conventions used in the proofs.

The construction in \Cref{sec:construction} forms the technical core of the paper. It has two main steps: constructing an $\ell_2$-sum $Y=(\bigoplus_{j=1}^{\infty}E_j)_2$ without DPR-lust, and realising an isomorphic copy $Z_\rho$ as a complemented subspace of $X_\rho$ with projection norm less than $1+\rho$. Both use the finite-dimensional tensor spaces $E_j\subseteq L_{p_j}$ introduced in \Cref{sec:tensors}, with suitably chosen exponents $p_j\downarrow2$.

To prove failure of DPR-lust, we must exclude uniformly controlled unconditional bases in every finite-dimensional superspace of $E_j$ inside $Y$. The tensor and fourth-moment estimates in \Cref{lem:packet-estimates,lem:fourth-moment} first give a more restricted obstruction: a subspace of dimension at most $3\dim E_j$ with a controlled unconditional basis has small Hilbert overlap with $E_j$, even after a Hilbert space is adjoined, as shown in \Cref{thm:overlap,prop:finite-obstruction}.

The crucial passage to arbitrary superspaces is the finite selection principle of \Cref{lem:finite-selection}. Using the trace estimate of \Cref{lem:trace-estimate}, it selects at most $3\dim E_j$ coordinates from a proposed basis with a controlled unconditional constant, so that the projection onto those coordinates retains a substantial trace after compression to $E_j$. The number of selected coordinates is independent of the dimension of the superspace.

In \Cref{sec:choosing-summands}, we recursively choose the subsequence $(E_j)_{j=1}^{\infty}$ of tensor spaces used to form $Y$. Together with \Cref{lem:local-hilbert}, this choice allows the selected complementary components to be replaced by vectors in a Hilbert space with controlled distortion. The Gaussian estimate then turns the retained trace into a lower bound for overlap. Comparing the two overlap bounds forces the unconditional constant to grow with $j$, uniformly over all such superspaces. This gives \Cref{thm:separated} and proves that $Y$ fails DPR-lust.

Complementation comes from tensorising the projections onto the two-dimensional function spaces defining $E_j$, as explained in \Cref{sec:tensors}. Near $p=2$, the first-order terms in the logarithm of the one-factor projection bound cancel. The choice of parameters in \Cref{sec:parameters} therefore makes the tensor projection norms tend to one while the obstruction to DPR-lust grows. The complemented embedding is then constructed in \Cref{sec:complemented-embedding}.

The embeddings in \Cref{sec:embedding} pass to finite-dimensional $\ell_p$ spaces while retaining the required projection bound. Over $\R$, \Cref{lem:sampling} uses exact covariance on a circle grid; over $\C$, \Cref{lem:complex-sampling} uses finite conditional expectations and a correction of an approximate inverse. The embeddings and their left inverses are uniformly bounded. The corresponding diagonal operators $A\colon Y\to X_\rho$ and $B\colon X_\rho\to Y$ satisfy $BA=I_Y$, so $Z_\rho=A(Y)$ is isomorphic to $Y$ and $P_\rho=AB$ is a projection of norm less than $1+\rho$. The ambient sum $X_\rho$ has its canonical $1$-unconditional basis.

In \Cref{sec:lattices}, we establish the GL/DPR separation and its dual consequences. The GL-lust bound follows from \Cref{lem:complemented-gl}, while the quotient version of the local Hilbert estimate yields the dual obstruction in \Cref{prop:dual-obstruction}. Over $\R$, a stronger recursion controlling the earlier ambient spaces gives two complementary summands, neither isomorphic to a Banach lattice, as established in \Cref{cor:two-nonlattice-ranges}.

Finally, \Cref{sec:geometry} constructs conditional asymptotically monotone bases for the examples and their duals, as stated in \Cref{cor:schauder-bases}. The section concludes with brief geometric consequences of the $\ell_2$-sum structure.

\section{Preliminaries and notation}\label{sec:preliminaries}

We follow the notation and terminology in \cite{LindenstraussTzafriri1977}, unless otherwise specified. We write $\mathbb K$ for the scalar field, which may be either $\R$ or $\C$. Unless stated otherwise, the definitions below apply over both fields. We treat the two scalar fields together in \Cref{sec:construction}, indicating where their proofs differ.

\subsection{Notation and unconditional constants}
Throughout, an \emph{operator} is a bounded linear map. For Banach spaces $E$ and $F$ over $\mathbb K$, we denote by $\mathscr B(E,F)$ the Banach space of operators from $E$ to $F$, equipped with the operator norm. When $E=F$, we write $\mathscr B(E)$ in place of $\mathscr B(E,E)$. The dual of $E$ is denoted by $E^*$, and $I_E$ denotes the identity operator on $E$. Given $x\in F$ and $f\in E^*$, we write $x\otimes f$ for the rank-one operator 
\begin{equation*}
 (x\otimes f)(y)=f(y)x\qquad(y\in E).
\end{equation*}

\Needspace{8\baselineskip}
We write $\N=\{1,2,\ldots\}$. For a family $(E_j)_{j\in J}$ of Banach spaces, the notation $(\bigoplus_{j\in J}E_j)_2$ denotes their \emph{$\ell_2$-sum}: its elements are the families $(x_j)_{j\in J}$, with $x_j\in E_j$ for $j\in J$ and
\begin{equation*}
 \norm{(x_j)_{j\in J}}
   =\left(\sum_{j\in J}\norm{x_j}_{E_j}^2\right)^{1/2}<\infty.
\end{equation*}

\Needspace{8\baselineskip}
Let $(b_i)_{i\in I}$ be a Schauder basis of a Banach space $E$. We use $I=\{1,\ldots,d\}$ when $E$ is finite-dimensional and $I=\N$ otherwise. We denote its coordinate functionals by $(b_i^*)_{i\in I}\subseteq E^*$. Its \emph{basis constant} is
\begin{equation*}
 b\bigl((b_i)_{i\in I}\bigr)
   =\sup_{m\in I}\left\|\sum_{i=1}^m b_i\otimes b_i^*\right\|,
\end{equation*}
the supremum of the norms of its initial coordinate projections. When $I=\N$, the basis is \emph{asymptotically monotone} if
\begin{equation*}
 \lim_{m\to\infty}\left\|\sum_{i=1}^m b_i\otimes b_i^*\right\|=1;
\end{equation*}
see \cite[Section~1]{MujicaVieira2010}.

\Needspace{7\baselineskip}
For $C\geq1$, we say that the basis is \emph{$C$-unconditional} if
\begin{equation*}
 \left\|\sum_{i\in I}\theta_i b_i\otimes b_i^*\right\|\leq C
\end{equation*}
for every finitely supported family $(\theta_i)_{i\in I}$ in $\mathbb K$ satisfying $|\theta_i|\leq1$ for every $i\in I$. In particular, the coordinate projection $\sum_{i\in M}b_i\otimes b_i^*$ has norm at most $C$ for every $M\subseteq I$.

For a finite-dimensional Banach space $E$, we denote by $u(E)$ the infimum of the unconditional basis constants over all bases of $E$. We shall use without further mention that rescaling the basis vectors by nonzero scalars leaves the unconditional basis constant unchanged.

\subsection{Hilbert norms and factorisation}\label{sec:hilbert-norms}
In addition to their $L_p$ norms, the tensor spaces of \Cref{sec:tensors} carry the Euclidean norm of their coefficient vectors; we call this the \emph{coefficient Hilbert norm} and denote it by $|\cdot|_2$. Hilbert tensor products are denoted by $\otimes_2$. Whenever we discuss adjoints, orthogonality or Hilbert--Schmidt norms on these spaces, we refer to their coefficient Hilbert structures. In particular, orthogonal projections need not be contractive for the Banach norm.

Our inner products are linear in the first variable. We write $A^{\mathsf T}$ for the transpose of a real matrix and $A^*$ for a Hilbert adjoint. For Hilbert spaces $H$ and $K$ with $\dim H=d<\infty$, we set
\begin{equation*}
 \langle A,B\rangle_{\mathrm{HS}}=\tr(B^*A),\qquad
 \norm T_{\mathrm{HS}}^2=\sum_{i=1}^d|Te_i|_2^2,
\end{equation*}
for $A,B,T\in\mathscr B(H,K)$, where $(e_i)_{i=1}^d$ is any orthonormal basis of $H$. The notation $\norm A_{2\to2}$ refers to the operator norm between the specified Hilbert spaces.

For isomorphic Banach spaces $E$ and $F$, we denote their \emph{Banach--Mazur distance} by
\begin{equation*}
 d_{\mathrm{BM}}(E,F)
 =\inf\{\norm J\norm{J^{-1}}:
             J\colon E\to F\text{ is an isomorphism}\}.
\end{equation*}
Given $D\geq1$, we say that $E$ is \emph{$D$-Hilbertian} if there is a Hilbert norm $|\cdot|_H$ on $E$ such that $|x|_H\leq\norm x_E\leq D|x|_H$ for every $x\in E$. Equivalently, $E$ admits an isomorphism $J$ onto a Hilbert space with $\norm J\leq1$ and $\norm{J^{-1}}\leq D$.

For $T\in\mathscr B(E,F)$, we define its \emph{Hilbert-factorisation norm} by
\begin{equation*}
 \gamma_2(T)=\inf\{\norm A\norm B:T=BA,\
                 A\in\mathscr B(E,H),\ B\in\mathscr B(H,F),\
                 H\text{ a Hilbert space}\},
\end{equation*}
with the convention that the infimum of the empty set is $+\infty$. The operators for which $\gamma_2$ is finite form a linear space on which $\gamma_2$ is a norm. We shall also use that $\gamma_2(STR)\leq\norm S\,\gamma_2(T)\norm R$.

\Needspace{10\baselineskip}
\subsection{Local unconditional structure}\label{subsec:local-unconditional}
We shall use two standard notions of local unconditional structure: DPR-lust and GL-lust. They differ in whether the finite-dimensional spaces with unconditional bases must lie inside the given space or may instead be auxiliary spaces in a factorisation. We now recall the definitions for a Banach space $Z$.

We say that $Z$ has \emph{local unconditional structure in the sense of Dubinsky, Pe\l czy\'nski and Rosenthal}, or \emph{DPR-lust} \cite{DubinskyPelczynskiRosenthal1972}, if there is a constant $C\geq1$ such that every nonzero finite-dimensional subspace $V\subseteq Z$ is contained in a finite-dimensional subspace $F\subseteq Z$ with a $C$-unconditional basis. To keep track of the constants, we put
\begin{equation*}
\begin{aligned}
 \lambda_{\mathrm{DPR}}(Z;V)
   &=\inf\{u(F):V\subseteq F\subseteq Z,\ \dim F<\infty\},\\
 \chi_{\mathrm{DPR}}(Z)
   &=\sup_{\substack{V\subseteq Z\\0<\dim V<\infty}}
          \lambda_{\mathrm{DPR}}(Z;V).
\end{aligned}
\end{equation*}
We call $\chi_{\mathrm{DPR}}(Z)$ the \emph{DPR constant} of $Z$. Thus $Z$ has DPR-lust precisely when this constant is finite.

For the Gordon--Lewis notion \cite{GordonLewis1974}, the finite-dimensional space with an unconditional basis may instead be an auxiliary space. Given a nonzero finite-dimensional $V\subseteq Z$, let $\iota_V\colon V\to Z$ denote the inclusion. We say that $Z$ has \emph{local unconditional structure in the sense of Gordon and Lewis}, or \emph{GL-lust}, if each $\iota_V$ factors through a finite-dimensional space $U$ with a $1$-unconditional basis, with the products of the norms of the two operators bounded independently of $V$. The \emph{GL constant} is
\begin{equation*}
 \chi_{\mathrm{GL}}(Z)
 =\sup_{\substack{V\subseteq Z\\0<\dim V<\infty}}
 \inf\{\norm a\norm b:
       V\xrightarrow{a}U\xrightarrow{b}Z,\ ba=\iota_V\}.
\end{equation*}
Here the infimum is taken over all finite-dimensional Banach spaces $U$ with a $1$-unconditional basis and all bounded linear maps $a\colon V\to U$ and $b\colon U\to Z$ satisfying the indicated identity. Thus $Z$ has GL-lust precisely when $\chi_{\mathrm{GL}}(Z)<\infty$. The map $b$ need not be injective; unlike the superspace $F$ in the DPR definition, $U$ need not be realised inside $Z$.

Every finite-dimensional space with a $C$-unconditional basis admits an equivalent norm, at distortion at most $C$, for which that basis is $1$-unconditional. Applying this observation to the superspaces in the DPR definition gives $\chi_{\mathrm{GL}}(Z)\leq\chi_{\mathrm{DPR}}(Z)$. Both properties are invariant under isomorphism, by applying an isomorphism to the superspaces or the factorisations, respectively. The infima need not be attained, so we allow an arbitrarily small increase in a constant when choosing a superspace or a factorisation.

For every nonzero real Banach lattice $L$, one has $\chi_{\mathrm{DPR}}(L)=1$ and hence $\chi_{\mathrm{GL}}(L)=1$ \cite[p.~397]{FigielJohnsonTzafriri1975}.

\subsection{Decompositions and approximation}
\label{subsec:decompositions} A sequence $(Z_j)_{j=1}^{\infty}$ of finite-dimensional subspaces of a Banach space $Z$ is a \emph{finite-dimensional decomposition} if every $z\in Z$ has a unique norm-convergent representation $z=\sum_{j=1}^{\infty}z_j$, with $z_j\in Z_j$ for every $j\geq1$. We call this decomposition \emph{contractive and unconditional} if, for every $m\geq1$ and every choice of vectors $(z_j)_{j=1}^m$ and scalars $(\theta_j)_{j=1}^m$ satisfying $z_j\in Z_j$ and $|\theta_j|\leq1$ for $1\leq j\leq m$, one has
\begin{equation*}
 \left\|\sum_{j=1}^{m}\theta_jz_j\right\|
   \leq\left\|\sum_{j=1}^{m}z_j\right\|.
\end{equation*}
Here the multipliers act on the summands of the decomposition. The condition places no uniform bound on the unconditional basis constants of the spaces $Z_j$.

The space $Z$ has the \emph{metric approximation property} if its identity operator can be approximated uniformly on compact subsets by finite-rank contractions. Clearly, a space with a contractive unconditional finite-dimensional decomposition has the metric approximation property, since its initial block projections are finite-rank contractions converging strongly to the identity.

\subsection{Probability conventions}\label{subsec:probability}
The tensor spaces of \Cref{sec:tensors} are defined on probability spaces. For $1<p<\infty$, we write $p'=p/(p-1)$ for the conjugate exponent. Finite-dimensional $\ell_p$ spaces carry the usual counting norm unless a normalised probability norm is explicitly specified.

For an integrable random variable $\xi$, its expectation is denoted by $\mathbb E[\xi]$. A subscript indicates the variables over which the expectation is taken: for example, $\mathbb E_t[\Phi(t,s)]$ means that $t$ is averaged and $s$ is held fixed. In the absence of a subscript, we average over all random variables occurring in the expression. By a sequence of \emph{Rademacher variables} $(\eps_i)_{i=1}^{\infty}$ we mean independent random variables taking the values $-1$ and $1$, each with probability $1/2$.

A \emph{standard scalar Gaussian} is a standard real Gaussian when $\mathbb K=\R$, and a variable $(g+ih)/\sqrt2$, with $g,h$ independent standard real Gaussians, when $\mathbb K=\C$. In both cases its second absolute moment is one.

Let $H$ be a Hilbert space over $\mathbb K$ of dimension $q\geq1$, and let $(e_i)_{i=1}^q$ be an orthonormal basis of $H$. A \emph{normalised Gaussian vector} in $H$ is a random vector of the form $\zeta=q^{-1/2}\sum_{i=1}^q g_i e_i$, where $(g_i)_{i=1}^q$ are independent standard scalar Gaussian variables. Its covariance operator is $I_H/q$. Consequently, for every linear operator $T$ from $H$ into a Hilbert space,
\begin{equation}\label{eq:gaussian-hs}
 \mathbb E[|\zeta|_2^2]=1,\qquad
 \norm T_{\mathrm{HS}}^2=q\,\mathbb E[|T\zeta|_2^2].
\end{equation}
\section{Construction of the examples}\label{sec:construction}

We work over $\mathbb K=\R$ or $\mathbb K=\C$ unless stated otherwise, indicating the differences between the two cases where necessary. Fix $\rho>0$, and write $X,P,Z$ for $X_\rho,P_\rho,Z_\rho$.

We construct an $\ell_2$-sum of finite-dimensional spaces without DPR-lust and realise it as a complemented subspace of a space with a $1$-unconditional basis, with projection norm less than $1+\rho$.

\subsection{Circle and sphere tensors}\label{sec:tensors}

For $\mathbb K=\R$, let $\Omega=\T=\R/(2\pi\mathbb Z)$ carry normalised Haar measure and put $w(t)=\sqrt2(\cos t,\sin t)$. For $\mathbb K=\C$, let $\Omega$ be the unit sphere of $\C^2$ with its invariant probability measure and put $w(t)=\sqrt2t$. In either case, $w$ is uniformly distributed on the sphere of radius $\sqrt2$ in $\mathbb K^2$. For each $n\geq1$, set
\begin{equation*}
 \mathcal H_n=(\mathbb K^2)^{\otimes_2 n},\qquad q_n=\dim\mathcal H_n=2^n.
\end{equation*}
For $t=(t_s)_{s=1}^n\in\Omega^n$, the tensor $W(t)=\bigotimes_{s=1}^nw(t_s)$ satisfies
\begin{equation}\label{eq:isotropy}
 |W(t)|_2^2=q_n,\qquad \mathbb E_t[W(t)W(t)^*]=I_{\mathcal H_n}.
\end{equation}
For $2<p\leq3$, we denote by $E(n,p)$ the space $\mathcal H_n$ with norm
\begin{equation*}
 \norm{x}_{E(n,p)}=\norm{f_x}_{L_p(\Omega^n)},\qquad f_x(t)=\langle x,W(t)\rangle.
\end{equation*}
We identify $E(n,p)$ with its image under $x\mapsto f_x$ in $L_p(\Omega^n)$. For $r>0$, put
\begin{equation}\label{eq:scalar-constants}
 \alpha_r=(\mathbb E[|w_1|^r])^{1/r},\qquad
 c_p=\alpha_p\alpha_{p'},
 \qquad \gamma_p=(\mathbb E[|g|^p])^{1/p},
\end{equation}
where $g$ is a standard scalar Gaussian as in \Cref{subsec:probability}. These constants depend on the fixed scalar field. Explicitly,
\begin{equation}\label{eq:explicit-alpha}
 \alpha_r=
 \begin{cases}
 \norm{\sqrt2\cos t}_{L_r(\T)},&\mathbb K=\R,\\
 \sqrt2(1+r/2)^{-1/r},&\mathbb K=\C,
 \end{cases}
\end{equation}
since the squared modulus of the first coordinate of a uniform unit vector in $\C^2$ is uniform on $[0,1]$. Throughout this subsection, $\zeta$ denotes a normalised Gaussian vector in $\mathcal H_n$.

\begin{lemma}\label{lem:packet-estimates}
Let $E=E(n,p)$. Then the following hold.
\begin{resultparts}
\item\label{it:packet-hilbert}
For every $x\in\mathcal H_n$,
\begin{equation*}
 |x|_2\leq\norm x_E\leq \alpha_p^n|x|_2.
\end{equation*}
\item\label{it:packet-products}
If $x$ is a tensor product of $n$ unit vectors, then
\begin{equation*}
 \norm x_E=\alpha_p^n,\qquad
 \norm{\langle\,\cdot\,,x\rangle}_{E^*}\leq\alpha_{p'}^n.
\end{equation*}
\item\label{it:packet-gaussian}
For the normalised Gaussian vector $\zeta$ and every contraction $A\colon\mathcal H_n\to\mathcal H_n$,
\begin{equation*}
 \mathbb E_\zeta[\norm \zeta_E^p]=\gamma_p^p,\qquad
 \mathbb E_\zeta[\norm{A\zeta}_E^2]\leq\gamma_p^2.
\end{equation*}
\end{resultparts}
\end{lemma}

\begin{proof}
For \ref{it:packet-hilbert}, isotropy gives $\norm{f_x}_2=|x|_2$, which proves the lower bound. To prove the upper bound, let $V\colon\mathbb K^2\to K$ be a linear map into a Hilbert space. Choose an orthonormal eigenbasis of $V^*V$, with corresponding eigenvalues $\lambda_1,\lambda_2$. By orthogonal or unitary invariance, the coordinates of $w$ in this basis have the same joint distribution as $(w_1,w_2)$. Hence
\begin{equation*}
 \mathbb E[|Vw|_2^p]
 =\mathbb E\left[(\lambda_1|w_1|^2+\lambda_2|w_2|^2)^{p/2}\right].
\end{equation*}
For fixed $\lambda_1+\lambda_2$, the right-hand side is convex in the eigenvalues and therefore attains its maximum when one of them is zero. Since $\lambda_1+\lambda_2=\tr(V^*V)=\norm V_{\mathrm{HS}}^2$, we obtain
\begin{equation*}
 \mathbb E[|Vw|_2^p]
 \leq\alpha_p^p(\lambda_1+\lambda_2)^{p/2}
 =\alpha_p^p\norm V_{\mathrm{HS}}^p.
\end{equation*}
Taking $p$th roots gives
\begin{equation*}
 (\mathbb E[|Vw(t)|_2^p])^{1/p}\leq\alpha_p\norm V_{\mathrm{HS}}.
\end{equation*}
For the tensor estimate, we argue by induction on $n$, the case $n=1$ following from the estimate above. For $n\geq2$, write $x=e_1\otimes x_1+e_2\otimes x_2$, where $(e_i)_{i=1}^2$ is the standard basis of $\mathbb K^2$ and $x_1,x_2\in\mathcal H_{n-1}$. With $t'=(t_s)_{s=2}^n$, we have
\begin{equation*}
 f_x(t)=\overline{w_1(t_1)}f_{x_1}(t')+\overline{w_2(t_1)}f_{x_2}(t').
\end{equation*}
Since $\overline w$ has the same distribution as $w$, we may apply the one-variable estimate in $t_1$. Integrating over $t'$ and then using Minkowski's inequality in $L_{p/2}(\Omega^{n-1})$ gives
\begin{equation*}
 \norm{f_x}_{L_p(\Omega^n)}
 \leq\alpha_p\left\|\left(\sum_{i=1}^2|f_{x_i}|^2\right)^{1/2}\right\|_{L_p(\Omega^{n-1})}
 \leq\alpha_p\left(\sum_{i=1}^2\norm{f_{x_i}}_{L_p(\Omega^{n-1})}^2\right)^{1/2}.
\end{equation*}
The induction hypothesis therefore yields
\begin{equation*}
 \norm x_E\leq\alpha_p^n\left(|x_1|_2^2+|x_2|_2^2\right)^{1/2}
 =\alpha_p^n|x|_2.
\end{equation*}

For \ref{it:packet-products}, let $x$ be a product unit vector. The scalar factors of $f_x$ have the same absolute moments as the first coordinate of $w$, so $\norm x_E=\alpha_p^n$. The dual estimate follows from isotropy and H\"older's inequality: $|\langle y,x\rangle|=|\mathbb E[f_y\overline{f_x}]| \leq\norm y_E\alpha_{p'}^n$.

To prove \ref{it:packet-gaussian}, observe that, for each fixed $t$, $\langle \zeta,W(t)\rangle$ is a scalar Gaussian of variance one, so Fubini's theorem gives the identity in \ref{it:packet-gaussian}. Replacing $\zeta$ by $A\zeta$ changes that variance to $q_n^{-1}|A^*W(t)|_2^2\leq1$. The same moment calculation, followed by the inequality between the second and $p$th moments, gives the remaining estimate since $p\geq2$.
\end{proof}

To complement $E(n,p)$ in $L_p(\Omega^n)$, we use the projection onto the span of $\overline{w_1},\overline{w_2}$ in each variable. In the real case, this is the first-harmonic projection. In either field it is
\begin{equation*}
 Qf(t)=\left\langle\mathbb E_s[f(s)w(s)],w(t)\right\rangle.
\end{equation*}
The coefficient map $f\mapsto\mathbb E_s[f(s)w(s)]$ from $L_p(\Omega)$ to $\ell_2^2$ has norm at most $\alpha_{p'}$, while the map $x\mapsto\langle x,w(\cdot)\rangle$ has norm $\alpha_p$, and hence $\norm Q\leq c_p$. Fubini's theorem gives the same bound when $Q$ acts on one variable of $L_p(\Omega^n)$. Composing these commuting projections gives a projection onto $E(n,p)$ of norm at most $c_p^n$. We denote this projection by $P_n$. In \Cref{sec:embedding}, we obtain a corresponding projection on a finite-dimensional $\ell_p$ space.

\subsubsection*{A trace estimate}
The symmetries of $E(n,p)$ also allow us to control the trace of an operator in terms of its Hilbert factorisation norm.

\begin{lemma}\label{lem:trace-estimate}
Every operator $T$ on $E(n,p)$ satisfies
\begin{equation}\label{eq:trace-estimate}
 \frac{|\tr T|}{q_n}\leq\frac{\gamma_p}{\alpha_p^n}\gamma_2(T).
\end{equation}
\end{lemma}

\begin{proof}
Let $E=E(n,p)$, and let $U$ be uniformly distributed over the finite group generated by sign changes and coordinate interchanges in the tensor factors. Its elements are isometries for both norms. Averaging over this group removes off-diagonal entries and equalises the diagonal entries; see \cite[Section~2.5]{AubrunMullerHermes2026}. Thus, for any product unit vector $x$,
\begin{equation*}
 \mathbb E_U[U^{-1}TU]=\frac{\tr T}{q_n}I_E,
 \qquad
 \mathbb E_U[(Ux)(Ux)^*]=I_{\mathcal H_n}/q_n=\mathbb E[\zeta\zeta^*].
\end{equation*}
Let $T=BA$ be a Hilbert factorisation, with $A\colon E\to H$ and $B\colon H\to E$. Applying \Cref{lem:packet-estimates} and using equality of the covariances, we obtain
\begin{equation*}
 \frac{|\tr T|}{q_n}\alpha_p^n
 \overset{\text{\ref{it:packet-products}}}{=}\left\|\mathbb E_U[U^{-1}BAUx]\right\|_E
 \leq\norm B\bigl(\mathbb E_U[|AUx|_H^2]\bigr)^{1/2}
 =\norm B\bigl(\mathbb E_\zeta[|A\zeta|_H^2]\bigr)^{1/2}
 \!\overset{\text{\ref{it:packet-gaussian}}}{\leq}\gamma_p\norm A\norm B.
\end{equation*}
Taking the infimum over Hilbert factorisations proves the estimate.
\end{proof}
\subsection{A fourth-moment estimate}\label{sec:fourth-moments}

We shall need a fourth-moment bound for the vectors of a basis and its Hilbert biorthogonal family. We obtain it from King's multiplicativity theorem \cite{King2003} and a Hilbert--Schmidt estimate.

Put $\tau=5/4$ over $\R$ and $\tau=10/9$ over $\C$. We define an operator on $M_2(\mathbb K)$ and its tensor power on $M_{q_n}(\mathbb K)$ by
\begin{equation*}
 \mathcal D(A)=
 \begin{cases}
  (\tr(A)I_{\R^2}+A+A^{\mathsf T})/4,&\mathbb K=\R,\\
  (\tr(A)I_{\C^2}+A)/3,&\mathbb K=\C,
 \end{cases}
 \qquad \mathcal D_n=\mathcal D^{\otimes n}.
\end{equation*}
For a uniform unit vector $u$ in $\mathbb K^2$, the fourth-moment identity takes the form
\begin{equation}\label{eq:circle-moment}
 \mathbb E[uu^*Auu^*]=\frac12\mathcal D(A).
\end{equation}
Over $\R$, this follows from $\mathbb E[u_1^4]=3/8$, $\mathbb E[u_1^2u_2^2]=1/8$ and the vanishing of odd moments. Over $\C$, the corresponding identities are $\mathbb E[|u_1|^4]=1/3$ and $\mathbb E[|u_1u_2|^2]=1/6$, together with phase invariance. The map $\mathcal D$ is Hilbert--Schmidt self-adjoint. We next estimate the action of $\mathcal D_n$ on positive rank-one matrices.

\begin{lemma}\label{lem:rank-one-moment}
For $b\in\mathcal H_n$,
\begin{equation*}
 \norm{\mathcal D_n(bb^*)}_{\mathrm{HS}}^2
 \leq(\tau/2)^n|b|_2^4.
\end{equation*}
\end{lemma}

\begin{proof}
For a linear map $\Phi\colon M_d(\C)\to M_d(\C)$, write
\begin{equation*}
 \nu_2(\Phi)=\sup\{\norm{\Phi(A)}_{\mathrm{HS}}:A\geq0,\ \tr A=1\}.
\end{equation*}
King's theorem \cite[Theorem~1]{King2003} gives $\nu_2(\Phi^{\otimes n})=\nu_2(\Phi)^n$ for maps of the form $\Phi(A)=\sum_{\ell=1}^{m}\tr(M_\ell A)R_\ell$, where $M_\ell,R_\ell\geq0$, $\tr R_\ell=1$ for $1\leq\ell\leq m$, and $\sum_{\ell=1}^{m}M_\ell=I_{\C^d}$.

In the real case, extend $\mathcal D$ complex-linearly to $M_2(\C)$. Let $e_1,e_2$ be the standard unit vectors in $\C^2$, and put
\begin{equation*}
 \begin{aligned}
 \mathcal U_{\R}&=\left\{e_1,e_2,\frac{e_1+e_2}{\sqrt2},
                                    \frac{e_1-e_2}{\sqrt2}\right\},\\
 \mathcal U_{\C}&=\mathcal U_{\R}\cup
                  \left\{\frac{e_1+ie_2}{\sqrt2},
                         \frac{e_1-ie_2}{\sqrt2}\right\}.
 \end{aligned}
\end{equation*}
With $P_u=uu^*$, the defining formulas give
\begin{equation*}
 \mathcal D(A)=\frac{2}{|\mathcal U_{\mathbb K}|}
                  \sum_{u\in\mathcal U_{\mathbb K}}\tr(P_uA)P_u,
 \qquad \frac{2}{|\mathcal U_{\mathbb K}|}
                  \sum_{u\in\mathcal U_{\mathbb K}}P_u=I_{\C^2}.
\end{equation*}
Thus King's theorem applies. For positive $A$ with $\tr A=1$, we have $\tr\mathcal D(A)=1$, and the eigenvalues of $\mathcal D(A)$ lie in $[1/4,3/4]$ when $\mathbb K=\R$ and in $[1/3,2/3]$ when $\mathbb K=\C$. Consequently $\nu_2(\mathcal D)^2=\tau/2$, with equality at $A=e_1e_1^*$. Multiplicativity and homogeneity complete the proof.
\end{proof}

\Needspace{10\baselineskip}
We now use this estimate to bound the fourth moment of a square function associated with two families of tensors.

\begin{lemma}\label{lem:fourth-moment}
Let $B,V\colon\ell_2^k(\mathbb K)\to\mathcal H_n$ have columns $(b_i)_{i=1}^k$ and $(g_i)_{i=1}^k$, respectively. Let $x$ be the tensor product of $n$ independent uniform unit vectors in $\mathbb K^2$, independent of a uniform $t\in\Omega^n$, and set
\begin{equation*}
 S(x,t)^2=\sum_{i=1}^k|\langle x,g_i\rangle|^2|f_{b_i}(t)|^2.
\end{equation*}
Then
\begin{align}
 \mathbb E[S^2]&=q_n^{-1}\sum_{i=1}^k|b_i|_2^2|g_i|_2^2,
                                      \label{eq:square-second}\\
 \mathbb E[S^4]&\leq\frac{k}{q_n}(\norm B\norm V)^4\tau^n.
                                      \label{eq:square-fourth}
\end{align}
The norms of $B$ and $V$ are their Hilbert operator norms.
\end{lemma}

\begin{proof}
The covariances $I_{\mathcal H_n}/q_n$ and $I_{\mathcal H_n}$ of $x$ and $W(t)$ give, for $i=1,\ldots,k$,
\begin{equation*}
 \mathbb E_x[|\langle x,g_i\rangle|^2]=q_n^{-1}|g_i|_2^2,
 \qquad \mathbb E_t[|f_{b_i}(t)|^2]=|b_i|_2^2.
\end{equation*}
Since $x$ and $t$ are independent, the expectation of each summand in $S^2$ is the product of these two expectations. Summing over $i=1,\ldots,k$ proves the first identity.

For the fourth-moment bound, let $R_i=b_ib_i^*$ for $i=1,\ldots,k$. For any matrix $A$,
\begin{equation*}
 \sum_{i=1}^k|\langle R_i,A\rangle_{\mathrm{HS}}|^2
 =\sum_{i=1}^k|(B^*AB)_{ii}|^2
 \leq\norm B^4\norm A_{\mathrm{HS}}^2.
\end{equation*}
Applying this inequality with $A=\mathcal D_n(R_j)$, summing over $j=1,\ldots,k$, and using \Cref{lem:rank-one-moment}, we obtain
\begin{equation}\label{eq:column-moments}
 \sum_{i,j=1}^k|\langle R_i,\mathcal D_n(R_j)\rangle_{\mathrm{HS}}|^2
 \leq\norm B^4(\tau/2)^n\sum_{j=1}^k|b_j|_2^4
 \leq k\norm B^8(\tau/2)^n.
\end{equation}
The same estimate holds for $g_ig_i^*$ with $V$ in place of $B$. To relate these estimates to $S$, we apply \eqref{eq:circle-moment} in each tensor factor, which gives
\begin{equation*}
 \mathbb E_x\bigl[|\langle x,g_i\rangle|^2|\langle x,g_j\rangle|^2\bigr]
 =q_n^{-1}\langle g_ig_i^*,
                     \mathcal D_n(g_jg_j^*)\rangle_{\mathrm{HS}}.
\end{equation*}
Since $W(t)$ has the distribution of $\sqrt{q_n}$ times an independent copy of $x$,
\begin{equation*}
 \mathbb E_t\bigl[|f_{b_i}(t)|^2|f_{b_j}(t)|^2\bigr]
 =q_n\langle R_i,\mathcal D_n(R_j)\rangle_{\mathrm{HS}}.
\end{equation*}
Expanding the square of the sum defining $S^2$ and using independence of $x$ and $t$, the preceding identities give
\begin{equation*}
 \begin{aligned}
 \mathbb E[S^4]
 &=\sum_{i,j=1}^k\langle R_i,\mathcal D_n(R_j)\rangle_{\mathrm{HS}}
       \langle g_ig_i^*,\mathcal D_n(g_jg_j^*)\rangle_{\mathrm{HS}}\\
 &\leq\left(\sum_{i,j=1}^k|\langle R_i,\mathcal D_n(R_j)\rangle_{\mathrm{HS}}|^2\right)^{1/2}
       \left(\sum_{i,j=1}^k|\langle g_ig_i^*,\mathcal D_n(g_jg_j^*)\rangle_{\mathrm{HS}}|^2\right)^{1/2}\\
 &\leq k(\norm B\norm V)^4(\tau/2)^n.
 \end{aligned}
\end{equation*}
Here the factors $q_n^{-1}$ and $q_n$ cancel, and the inequalities follow successively from Cauchy--Schwarz and the bounds in \eqref{eq:column-moments} for the columns of $B$ and $V$. Since $q_n=2^n$, this proves \eqref{eq:square-fourth}.
\end{proof}
\subsection{An obstruction for arbitrary bases}\label{sec:overlap}

Let $E=E(n,p)$ and let $H$ be a finite-dimensional Hilbert space. For $z=(x,h)\in E\oplus_2H$, we write $N(z)=N(x,h)=(\norm x_E^2+|h|_2^2)^{1/2}$ for its norm. When we consider the same vector space with the Hilbert norm $|z|_2=(|x|_2^2+|h|_2^2)^{1/2}$, we refer to this as its \emph{coefficient Hilbert structure}, denoted by $\mathcal H_n\oplus_2H$. Here $|x|_2$ is the Euclidean norm of the tensor coefficients of $x$. For a subspace $F\subseteq E\oplus_2H$, let $R_F$ be the orthogonal projection onto $F$ in this Hilbert structure, and let $\iota_E\colon E\to E\oplus_2H$ be the canonical inclusion. We define the \emph{overlap} of $F$ with $E$ by
\begin{equation}\label{eq:overlap-definition}
 \theta(F)=q_n^{-1}\tr(\iota_E^*R_F\iota_E)\in[0,1].
\end{equation}
The quantity $\theta(F)$ measures the overlap of $E$ and $F$ in the coefficient Hilbert structure: it is the average squared length of the orthogonal projections onto $F$ of the vectors in an orthonormal basis of $E$. It equals $1$ precisely when $E\subseteq F$ and $0$ precisely when $E\perp F$.

The next theorem bounds the overlap of $F$ with $E$ in terms of their relative dimensions and the unconditional basis constant of $F$, independently of the dimension of $H$. In the estimate below, we write $L_0=(\gamma_p^2+1)^{1/2}$ and $\beta_p=3^{(p-2)/(2p)}$.

\begin{theorem}\label{thm:overlap}
Let $M\in\N$ and $C\geq1$. Suppose $F\subseteq E\oplus_2H$ has dimension at most $Mq_n$ and a $C$-unconditional basis. Then
\begin{equation*}
 \theta(F)\alpha_{p'}^{-n}
 \leq\beta_p C^{4-4/p}L_0^{(4-p)/p}M^{(p-2)/(2p)}
       \alpha_p^{4n(p-2)/p}\tau^{n(p-2)/(2p)} +C^2L_0.
\end{equation*}
The bound is independent of the dimension of $H$.
\end{theorem}

\begin{proof}
We may assume $k=\dim F>0$. We first compare the given basis with the coefficient Hilbert structure. After normalising the given $C$-unconditional basis $(b_i)_{i=1}^k$ of $F$ so that $|b_i|_2=1$ for $i=1,\ldots,k$, let $(g_i)_{i=1}^k$ be its Hilbert biorthogonal family in $F$. Let $B\colon\ell_2^k\to F$ be given by $B((a_i)_{i=1}^k)=\sum_{i=1}^k a_i b_i$, and let $D_\eps$ be the basis multiplier with signs $\eps=(\eps_i)_{i=1}^k\in\{-1,1\}^k$. The inequalities $|z|_2\leq N(z)\leq \alpha_p^n|z|_2$ imply $\norm{D_\eps}_{2\to2}\leq C\alpha_p^n$. The normalisation of the basis gives, for $a=(a_i)_{i=1}^k\in\ell_2^k$,
\begin{equation*}
 \mathbb E_\eps\bigl[|B((\eps_i a_i)_{i=1}^k)|_2^2\bigr]=|a|_2^2.
\end{equation*}
Since each sign multiplier is its own inverse, the preceding bound applied in both directions yields
\begin{equation}\label{eq:synthesis-bounds}
 \norm B_{2\to2}\leq C\alpha_p^n,\qquad \norm{B^{-1}}_{2\to2}\leq C\alpha_p^n.
\end{equation}
The corresponding operator for $(g_i)_{i=1}^k$ is $(B^{-1})^*$ and therefore satisfies the same bound.

Let $x$ be a product unit vector as in \Cref{lem:fourth-moment}, and put $z=R_F(x,0)$. By \Cref{lem:packet-estimates}\ref{it:packet-products}, the functional represented by $(x,0)$ has dual norm at most $\alpha_{p'}^n$. Since $R_F$ is an orthogonal projection, this gives
\begin{equation*}
 0\leq\langle R_F(x,0),(x,0)\rangle=|z|_2^2\leq\alpha_{p'}^nN(z).
\end{equation*}
Taking expectations and using $\mathbb E_x[xx^*]=I_{\mathcal H_n}/q_n$ to evaluate the quadratic form, we obtain
\begin{equation}\label{eq:overlap-lower}
 \mathbb E_x[N(z)]\geq\alpha_{p'}^{-n}
    \mathbb E_x\bigl[\langle R_F(x,0),(x,0)\rangle\bigr]
 =\frac{\alpha_{p'}^{-n}}{q_n}\tr(\iota_E^*R_F\iota_E)
 =\theta(F)\alpha_{p'}^{-n}.
\end{equation}

To obtain an upper bound, we randomise the expansion of $z$ in the given basis. For $i=1,\ldots,k$, write $b_i=(b_i^E,b_i^H)$ and $g_i=(g_i^E,g_i^H)$ according to the two summands, and define
\begin{equation*}
 \begin{aligned}
 S_E(x,t)^2&=\sum_{i=1}^k|\langle x,g_i^E\rangle|^2
                                      |f_{b_i^E}(t)|^2,\\
 V_H(x)&=\sum_{i=1}^k|\langle x,g_i^E\rangle|^2|b_i^H|_2^2.
 \end{aligned}
\end{equation*}
Both expressions are quadratic in $x$, so their expectations are unchanged if $x$ is replaced by a normalised Gaussian vector $\zeta$. We obtain
\begin{equation*}
 \mathbb E[S_E^2]+\mathbb E[V_H]
 =\mathbb E_{\zeta,\eps}\bigl[|D_\eps R_F(\zeta,0)|_2^2\bigr]
 \leq C^2\mathbb E_\zeta\bigl[N(R_F(\zeta,0))^2\bigr],
\end{equation*}
where the equality follows by averaging over the independent signs and using $\mathbb E_t[|f_{b_i^E}(t)|^2]=|b_i^E|_2^2$ for $i=1,\ldots,k$. The inequality follows from $|D_\eps u|_2\leq N(D_\eps u)\leq C N(u)$ for $u\in F$, by unconditionality.

The operator $\iota_E^*R_F\iota_E$ is a Hilbert contraction, so \Cref{lem:packet-estimates}\ref{it:packet-gaussian} bounds the expected squared $E$-norm of the first component of $R_F(\zeta,0)$ by $\gamma_p^2$. The Hilbert component contributes at most $\mathbb E[|\zeta|_2^2]=1$. Hence
\begin{equation}\label{eq:good-second-moment}
 \mathbb E[S_E^2]+\mathbb E[V_H]\leq C^2L_0^2.
\end{equation}
Let $B_E,V_E\colon\ell_2^k\to\mathcal H_n$ be the operators with columns $(b_i^E)_{i=1}^k$ and $(g_i^E)_{i=1}^k$, respectively. These are the $E$-components of $B$ and $(B^{-1})^*$. Since the coordinate projection onto $E$ is contractive for the coefficient Hilbert norm, \eqref{eq:synthesis-bounds} gives $\norm{B_E},\norm{V_E}\leq C\alpha_p^n$. Applying \Cref{lem:fourth-moment} to these operators and the product unit vector $x$, and using $k/q_n\leq M$, we obtain
\begin{equation*}
 \mathbb E_{x,t}[S_E^4]
 \leq\frac{k}{q_n}(\norm{B_E}\norm{V_E})^4\tau^n
 \leq M(C\alpha_p^n)^8\tau^n
 =MC^8\alpha_p^{8n}\tau^n.
\end{equation*}
Since $p\in(2,3]$, interpolating between this bound and \eqref{eq:good-second-moment} yields
\begin{equation}\label{eq:interpolated-moment}
 (\mathbb E[S_E^p])^{1/p}
 \leq(C^2L_0^2)^{(4-p)/(2p)}
       \bigl(MC^8\alpha_p^{8n}\tau^n\bigr)^{(p-2)/(2p)}.
\end{equation}
We observe that the Hilbertian distortion $\alpha_p^n$ appears in this bound only to the power $4(p-2)/p$, which tends to zero as $p\downarrow2$. This permits the choice of parameters in \Cref{sec:parameters}, for which the distortion tends to infinity while its contribution $\alpha_p^{4n(p-2)/p}$ to the upper bound tends to one.

Let $(\eps_i)_{i=1}^k$ be independent Rademacher variables. For scalars $(a_i)_{i=1}^k\in\mathbb K^k$, the sum $\sum_{i=1}^k\eps_i a_i$ has second moment $\sum_{i=1}^k|a_i|^2$ and fourth moment at most $3(\sum_{i=1}^k|a_i|^2)^2$, as follows by expansion. Interpolating between these moments gives
\begin{equation*}
 \left(\mathbb E_\eps\left[\left|\sum_{i=1}^k\eps_i a_i\right|^p\right]\right)^{1/p}
 \leq\beta_p\left(\sum_{i=1}^k|a_i|^2\right)^{1/2}.
\end{equation*}
Since $D_\eps^{-1}=D_\eps$, unconditionality implies $N(z)\leq C\mathbb E_\eps[N(D_\eps z)]$. For fixed $x$, write $(u_\eps,h_\eps)=D_\eps z=\sum_{i=1}^k\eps_i\langle x,g_i^E\rangle b_i$. Then
\begin{equation*}
 \mathbb E_\eps[N(D_\eps z)]
 \leq\bigl(\mathbb E_\eps[\norm{u_\eps}_E^p]\bigr)^{1/p}
       +\bigl(\mathbb E_\eps[|h_\eps|_2^2]\bigr)^{1/2}
 \leq\beta_p\bigl(\mathbb E_t[S_E(x,t)^p]\bigr)^{1/p}+V_H(x)^{1/2},
\end{equation*}
where the first inequality follows from $N(u,h)\leq\norm u_E+|h|_2$ and Jensen's inequality. For the second, we apply the preceding scalar estimate pointwise in $t$ and integrate, while $\mathbb E_\eps[|h_\eps|_2^2]=V_H(x)$ gives the Hilbert term. Averaging over $x$ and using Jensen's inequality once more gives
\begin{equation}\label{eq:randomisation-bound}
 \mathbb E_x[N(z)]
 \leq C\left[\beta_p(\mathbb E_{x,t}[S_E^p])^{1/p}
                         +(\mathbb E_x[V_H])^{1/2}\right].
\end{equation}
The estimates \eqref{eq:good-second-moment} and \eqref{eq:interpolated-moment} now give the required upper bound. Comparing this with \eqref{eq:overlap-lower} proves the theorem.
\end{proof}
\subsection{Choosing the finite-dimensional parameters}\label{sec:parameters}

Using the dependence on $\alpha_p^n$ in \eqref{eq:interpolated-moment}, we choose the exponents so that the projection norms tend to one while the obstruction to unconditional bases becomes arbitrarily large. For $n\geq1$, let
\begin{equation}\label{eq:parameters}
 p_n=2+n^{-3/4},\qquad E_n=E(n,p_n).
\end{equation}
To isolate $\theta(F)$ in \Cref{thm:overlap}, we multiply by $\alpha_{p_n'}^n=c_{p_n}^n/\alpha_{p_n}^n$. The resulting bound contains the ratio $\tau^{n(p_n-2)/(2p_n)}/\alpha_{p_n}^n$, so we set
\begin{equation}\label{eq:separation-scale}
 \Lambda_n=\alpha_{p_n}^n\tau^{-n(p_n-2)/(2p_n)}.
\end{equation}
For fixed $C$ and $M$, the remaining factors are bounded, as verified below, so $\Lambda_n^{-1}$ controls the overlap. We have $\Lambda_n\leq \alpha_{p_n}^n$; we shall see that both tend to infinity as $n\to\infty$.

In the real case, evaluating the $L_r$ norm in \eqref{eq:explicit-alpha} by the beta integral gives
\begin{equation*}
 \alpha_r^r=\frac{2^{r/2}\Gamma((r+1)/2)}
                   {\sqrt\pi\Gamma(1+r/2)}.
\end{equation*}
Together with the explicit formula for $\alpha_r$ in the complex case, this yields
\begin{equation*}
 \left.\frac{d}{dr}\log\alpha_r\right|_{r=2}
 =\begin{cases}
 (1-\log2)/4,&\mathbb K=\R,\\
 (\log2-1/2)/4,&\mathbb K=\C.
 \end{cases}
\end{equation*}
Writing $\eps=p-2$, we have $p'=2-\eps+O(\eps^2)$, so the first-order terms cancel in $\log c_p$. For our choice $\eps=p_n-2=n^{-3/4}$, the relations $n\eps=n^{1/4}\to\infty$ and $n\eps^2=n^{-1/2}\to0$ therefore give
\begin{equation}\label{eq:parameter-limits}
 c_{p_n}^n\longrightarrow1,\qquad
 \alpha_{p_n}^n\longrightarrow\infty,\qquad
 \alpha_{p_n}^{4n(p_n-2)/p_n}\longrightarrow1.
\end{equation}
The same expansion applied to $\Lambda_n$ gives
\begin{equation}\label{eq:positive-exponent}
 \log\Lambda_n=\kappa n^{1/4}+O(n^{-1/2}),\qquad
 \kappa=
 \begin{cases}
 (1-\log(5/2))/4,&\mathbb K=\R,\\
 (\log(9/5)-1/2)/4,&\mathbb K=\C.
 \end{cases}
\end{equation}
Both values of $\kappa$ are positive; for the complex value, use $e<3<(9/5)^2$.

We collect the Gaussian, trace and overlap estimates needed for the infinite sum. The last estimate shows that subspaces of dimension at most $3q_n$ with a uniformly bounded unconditional basis constant have overlap with $E_n$ tending to zero, independently of the dimension of the adjoined Hilbert space.

\begin{proposition}\label{prop:finite-obstruction}
For all sufficiently large $n$, independently of $C$ and $H$, the following hold:
\begin{resultparts}
\item\label{it:finite-gaussian}
For a normalised Gaussian vector $\zeta$ in $\mathcal H_n$,
\begin{equation*}
 \mathbb E_\zeta[\norm{\zeta}_{E_n}^2]\leq4.
\end{equation*}
\item\label{it:finite-trace}
For every $T\in\mathscr B(E_n)$,
\begin{equation*}
 \frac{|\tr T|}{q_n}\leq\frac2{\alpha_{p_n}^n}\gamma_2(T).
\end{equation*}
\item\label{it:finite-overlap}
Let $C\geq1$ and let $H$ be a finite-dimensional Hilbert space. If $F\subseteq E_n\oplus_2H$ has dimension at most $3q_n$ and a $C$-unconditional basis, then
\begin{equation*}
 \theta(F)\leq\frac{80C^3}{\Lambda_n}.
\end{equation*}
\end{resultparts}
\end{proposition}

\begin{proof}
\ref{it:finite-gaussian} and \ref{it:finite-trace} follow respectively from \Cref{lem:packet-estimates}\ref{it:packet-gaussian} and \Cref{lem:trace-estimate}, since $\gamma_{p_n}\to1$.

For \ref{it:finite-overlap}, we apply \Cref{thm:overlap} with $M=3$ and multiply by $\alpha_{p_n'}^n=c_{p_n}^n/\alpha_{p_n}^n$. The fourth-moment factor divided by $\alpha_{p_n}^n$ is $1/\Lambda_n$, while the remaining scalar factors in both terms tend to $\sqrt2$ by \eqref{eq:parameter-limits}. Since $4-4/p_n\leq3$, $C\geq1$ and $\Lambda_n\leq \alpha_{p_n}^n$, the estimate follows for all sufficiently large $n$, independently of $C$ and $H$.
\end{proof}

Taking $F=E_n$ and $H=\{0\}$ gives $\theta(F)=1$, hence $u(E_n)\geq(\Lambda_n/80)^{1/3}\to\infty$. This estimate concerns bases of the individual spaces. An unconditional basis of an infinite sum of these spaces need not have subsequences spanning the summands, so this estimate alone does not exclude such a basis. The overlap bound in \Cref{prop:finite-obstruction}\ref{it:finite-overlap} also applies to subspaces of $E_n\oplus_2H$ that need not contain $E_n$. The selection principle in \Cref{lem:finite-selection} will select a small coordinate subspace from a proposed basis and, after replacing its complementary components by vectors in a Hilbert space, give a lower bound on its overlap with $E_n$. The recursive choice of summands then makes the two overlap bounds incompatible.
\subsection{An \texorpdfstring{$\ell_2$}{ell2}-sum without DPR-lust}
\label{sec:infinite-sum}

We now use these estimates to construct a space without DPR-lust. The main step is the finite selection principle of \Cref{lem:finite-selection}, which, starting from an unconditional basis, selects a finite-dimensional coordinate subspace and gives a lower bound for its overlap with a prescribed summand. Together with \Cref{prop:finite-obstruction}, this forces every unconditional basis of a finite-dimensional superspace of that summand to have a large constant.

\subsubsection{Local Hilbert approximation}

The next lemma shows that subspaces of fixed dimension in $\ell_2$-sums of subspaces of $L_r$ become uniformly close to Hilbert space as the exponents approach two. This remains true for subspaces of quotients of these sums.

\begin{lemma}\label{lem:local-hilbert}
For every integer $k\geq1$ and $D>1$ there is $\varepsilon(k,D)>0$ with the following property. Let $A$ be a nonempty index set. For each $\alpha\in A$, let $r_\alpha\in[2,2+\varepsilon(k,D)]$ and let $X_\alpha$ be a closed subspace of an $L_{r_\alpha}$ space, and set
\begin{equation*}
    S=\left(\bigoplus_{\alpha\in A}X_{\alpha}\right)_2.
\end{equation*}
Then every subspace of dimension at most $k$ of every quotient of $S$ is $D$-Hilbertian.
\end{lemma}

\begin{proof}
For $r\geq2$ and $x,y\in L_r$, Clarkson's computation of the von Neumann--Jordan constant \cite{Clarkson1937} gives
\begin{equation}\label{eq:approximate-parallelogram}
    \nu_r^{-1}(\norm{x}_r^2+\norm{y}_r^2)
    \leq\frac{\norm{x+y}_r^2+\norm{x-y}_r^2}{2}
    \leq \nu_r(\norm{x}_r^2+\norm{y}_r^2),
    \qquad \nu_r=2^{1-2/r}.
\end{equation}
Here the left inequality follows by applying the right one to $x+y$ and $x-y$. Summing squared norms shows that the same inequalities hold in $S$, with constant $\nu=\sup_{\alpha\in A}\nu_{r_\alpha}$.

To obtain the same inequalities in quotients of $S$, we use the \emph{normalised Hadamard map} on $S\oplus_2 S$,
\begin{equation*}
 \mathsf H(x,y)=2^{-1/2}(x+y,x-y),
\end{equation*}
which has norm at most $\sqrt\nu$ and preserves $N\oplus_2N$ for every closed subspace $N\subseteq S$. It induces the same formula, with no larger norm, on $(S/N)\oplus_2(S/N)$. The induced map squares to $I_{(S/N)\oplus_2(S/N)}$, so both inequalities follow. They pass to subspaces as well, and their constants tend to one as all the exponents tend to two.

Passer's approximate Jordan--von Neumann theorem \cite[Theorem~3.4]{Passer2015} states that, in each fixed real dimension, the Banach--Mazur distance to Hilbert space tends to one as the constant in \eqref{eq:approximate-parallelogram} tends to one. Over $\R$, applying this result in dimensions $2,\ldots,k$, and observing that the one-dimensional case is automatic, proves the result. Over $\C$, we apply the same theorem in real dimensions at most $2k$. It gives a real Hilbert norm $|\cdot|_0$ satisfying $|x|_0\leq\norm x\leq D|x|_0$. The averaged norm
\begin{equation*}
 |x|_H^2=\frac{|x|_0^2+|ix|_0^2}{2}
\end{equation*}
satisfies the same comparison and is invariant under multiplication by $i$, so it comes from a complex inner product.
\end{proof}

Before we move on, we observe that the lemma applies in particular to subspaces of $S$ itself of dimension at most $k$. Moreover, the choice of $\varepsilon(k,D)$ may depend on the scalar field.

\subsubsection{A finite selection principle}

For the selection lemma, let $E$ be a $q$-dimensional Banach space with a specified Hilbert norm $|\cdot|_2$ satisfying $|x|_2\leq\norm{x}_E$. We use the same notion of overlap as in \Cref{sec:overlap}: if $H$ is a finite-dimensional Hilbert space and $F\subseteq E\oplus_2H$, we write
\begin{equation*}
    \theta_E(F)=\frac1q\tr(\iota_E^*R_F\iota_E),
\end{equation*}
where $R_F$ is the orthogonal projection in the coefficient Hilbert structure and $\iota_E$ is the inclusion of the first summand.

The next lemma is the key step in passing from the finite-dimensional overlap estimate to failure of DPR-lust. Under its hypotheses, it selects at most $3q$ vectors from an arbitrary unconditional basis and, after a Hilbertian replacement of their complementary components, produces a subspace of $E\oplus_2H$ with a controlled unconditional basis constant and a quantitative lower bound on its overlap with $E$. In \Cref{thm:separated}, comparison with the upper bound in \Cref{prop:finite-obstruction}\ref{it:finite-overlap} forces the unconditional basis constants of all finite-dimensional superspaces of the chosen summands to grow. This is what allows us to rule out DPR-lust even though these bases need not respect the given decomposition.

\Needspace{13\baselineskip}
\begin{samepage}
\begin{lemma}\label{lem:finite-selection}\label{prop:selection}
Suppose that $Y=E\oplus_2W$ has a $C$-unconditional basis and that every subspace of $W$ of dimension at most $3q$ is $D$-Hilbertian. Suppose that there are constants $a,L>0$ such that
\begin{equation}\label{eq:selection-hypotheses}
    \frac{|\tr T|}{q}\leq a\gamma_2(T)
       \quad(T\in\mathscr B(E)),
    \qquad \mathbb E[\norm{\zeta}_E^2]\leq L^2,
\end{equation}
where $\zeta$ is the normalised Gaussian vector in $E$. If $\eta=aDC^2\leq1/8$, there are a finite-dimensional Hilbert space $H$ and a subspace $F'\subseteq E\oplus_2H$ such that
\begin{resultparts}
\item\label{it:selection-dimension} $\dim F'\leq3q$.
\item\label{it:selection-basis} $F'$ has a $DC$-unconditional basis.
\item\label{it:selection-overlap} $\displaystyle\theta_E(F')\geq\frac{(1-4\eta)^2}{C^2L^2}$.
\end{resultparts}
\end{lemma}
\end{samepage}

\begin{proof}
Let $(z_i,z_i^*)_{i\in I}$ be the basis and its coordinate functionals, where $I=\mathbb N=\{1,2,\ldots\}$ in the infinite-dimensional case and $I=\{1,\ldots,m\}$ if $\dim Y=m<\infty$. Let $\iota_E\colon E\to Y$ be the canonical inclusion, and let $P_E\colon Y\to E$ and $P_W\colon Y\to W$ be the coordinate projections associated with $Y=E\oplus_2W$. Put
\begin{equation*}
    Q_i=z_i\otimes z_i^*,\qquad
    A_i=P_EQ_i\iota_E,\qquad t_i=\tr A_i\quad(i\in I).
\end{equation*}
Each $A_i$ has rank at most one, so $\tr(A_i^2)=t_i^2$. The partial sums of the $Q_i$ converge to the identity uniformly on the unit ball of the finite-dimensional space $\iota_E(E)$. Thus $\sum_{i\in I}t_i=q$.

For a finite set $J\subseteq I$, scalars $(w_i)_{i\in J}$ with $|w_i|\leq1$ for $i\in J$, and independent Rademacher signs $(\varepsilon_i)_{i\in J}$, put
\begin{equation*}
 M_\varepsilon=\sum_{i\in J}\varepsilon_iQ_i,
 \qquad
 N_\varepsilon=\sum_{i\in J}\varepsilon_iw_iQ_i.
\end{equation*}
Both operators have norm at most $C$. Write $T_\varepsilon,S_\varepsilon$ for their compressions to $E$, and $B_\varepsilon,V_\varepsilon$ for the upper-right block of $M_\varepsilon$ and the lower-left block of $N_\varepsilon$, respectively. Since $Q_iQ_k=\delta_{ik}Q_i$,
\begin{equation*}
 T_\varepsilon S_\varepsilon+B_\varepsilon V_\varepsilon
   =\sum_{i\in J}w_iA_i.
\end{equation*}
The range of $V_\varepsilon$ has dimension at most $q$ and is $D$-Hilbertian, so $\gamma_2(B_\varepsilon V_\varepsilon)\leq DC^2$. Taking traces and averaging over the signs therefore gives
\begin{equation*}
 \left|\sum_{i\in J}w_i(t_i-t_i^2)\right|
 =\left|\mathbb E[\tr(B_\varepsilon V_\varepsilon)]\right|
 \leq\eta q.
\end{equation*}
Choosing $w_i$ so that $w_i(t_i-t_i^2)=|t_i-t_i^2|$ and passing to increasing finite sets yields, in particular,
\begin{equation}\label{eq:trace-defect-bound}
    \sum_{i\in I}|t_i-t_i^2|\leq2\eta q.
\end{equation}

Set
\begin{equation*}
    S=\{i\in I:|t_i-1|\leq1/2\},\qquad
    Q_S=\sum_{i\in S}Q_i,\qquad F=Q_SY.
\end{equation*}
The set $S$ is finite, since $t_i\to0$ when $I=\mathbb N$. Outside $S$, the inequality $|t_i|\leq2|t_i-t_i^2|$ holds. Hence
\begin{equation}\label{eq:selected-trace}
    (1-4\eta)q\leq\operatorname{Re}\sum_{i\in S}t_i\leq(1+4\eta)q.
\end{equation}
For every $i\in S$, we have $\operatorname{Re}t_i\geq1/2$, so $|S|\leq2(1+4\eta)q\leq3q$. Moreover, $\norm{Q_S}\leq C$, and the selected vectors $(z_i)_{i\in S}$ form a $C$-unconditional basis of $F$.

It remains to replace the $W$-component of $F$ by a Hilbert space. Since $\dim P_WF\leq3q$, there is an isomorphism $J\colon P_WF\to H$ onto a Hilbert space such that $\norm J\leq1$ and $\norm{J^{-1}}\leq D$. Define $U\colon E\oplus_2P_WF\to E\oplus_2H$ by $U(x,w)=(x,Jw)$ and let $F'=UF$. This space has the same dimension as $F$, which proves \ref{it:selection-dimension}. The image under $U$ of the selected basis is $DC$-unconditional, proving \ref{it:selection-basis}. To prove \ref{it:selection-overlap}, define
\begin{equation*}
    T_0=UQ_S\iota_E\colon E\longrightarrow F'.
\end{equation*}
Write $\iota_E$ also for the inclusion into $E\oplus_2H$. The norm of $T_0\colon E\to F'$ is at most $C$, and $U$ fixes the $E$ component, so \eqref{eq:selected-trace} yields
\begin{equation*}
    \operatorname{Re}\tr(\iota_E^*T_0)\geq(1-4\eta)q.
\end{equation*}
The coefficient Hilbert norm on $E\oplus_2H$ is bounded above by its Banach norm. Therefore
\begin{equation*}
    \norm{T_0}_{\mathrm{HS}}^2
       =q\,\mathbb E[|T_0\zeta|_2^2]\leq qC^2L^2.
\end{equation*}
Let $R$ be the orthogonal projection onto $F'$. As $RT_0=T_0$, Cauchy--Schwarz for the Hilbert--Schmidt inner product gives
\begin{equation*}
    (1-4\eta)q
       \leq\operatorname{Re}\langle T_0,R\iota_E\rangle_{\mathrm{HS}}
       \leq\norm{T_0}_{\mathrm{HS}}\norm{R\iota_E}_{\mathrm{HS}}
       \leq CLq\,\theta_E(F')^{1/2}.
\end{equation*}
This proves \ref{it:selection-overlap} and completes the proof.
\end{proof}

\subsubsection{Choosing the summands}\label{sec:choosing-summands}

We return to the spaces $E_n\subseteq L_{p_n}$ and the constants $\Lambda_n$ from \Cref{prop:finite-obstruction}, with $q_n=\dim E_n$. Recall that
\begin{equation*}
    p_n\downarrow2,\qquad
    \Lambda_n=\alpha_{p_n}^n\tau^{-n(p_n-2)/(2p_n)}\longrightarrow\infty,
    \qquad \Lambda_n\leq \alpha_{p_n}^n.
\end{equation*}
We shall recursively choose a strictly increasing sequence $(n_j)_{j=1}^{\infty}$ of integers, all sufficiently large for the conclusions of \Cref{prop:finite-obstruction} to hold. To simplify notation, write
\begin{equation*}
    E_j=E_{n_j},\quad p_j=p_{n_j},\quad q_j=q_{n_j},\quad
    \Lambda_j=\Lambda_{n_j}\qquad(j\geq1),
\end{equation*}
and put $D_1=2$ and $D_j=\max\{2,\alpha_{p_{j-1}}^{n_{j-1}}\}$ for $j\geq2$. At each step, we also arrange that
\begin{equation}\label{eq:recursive-separation}
    \Lambda_j\geq D_j^8,
    \qquad p_j-2\leq\varepsilon(3q_i,2)\quad(1\leq i<j).
\end{equation}
Once $n_1,\ldots,n_{j-1}$ have been chosen, $D_j$ and the finitely many values $\varepsilon(3q_i,2)$ for $1\leq i<j$ are fixed. Since $p_n\to2$ and $\Lambda_n\to\infty$, these conditions can be satisfied by taking $n_j$ sufficiently large. These choices also make $(\alpha_{p_j}^{n_j})_{j=1}^{\infty}$ increasing, because $\alpha_{p_j}^{n_j}\geq\Lambda_j\geq D_j^8$.

Set
\begin{equation}\label{eq:constructed-range}
    Y=\left(\bigoplus_{j=1}^{\infty}E_j\right)_2,
    \qquad W_j=\left(\bigoplus_{\substack{i=1\\i\ne j}}^{\infty}E_i\right)_2
    \quad(j\geq1).
\end{equation}
Every subspace $V\subseteq W_j$ of dimension at most $3q_j$ is $D_j$-Hilbertian. Indeed, its projection onto the later summands is $2$-Hilbertian by \Cref{lem:local-hilbert} and \eqref{eq:recursive-separation}. For $v\in V$, write $v=v_-+v_+$, where $v_-$ and $v_+$ are its coordinate projections onto $(\bigoplus_{i=1}^{j-1}E_i)_2$ and $(\bigoplus_{i=j+1}^{\infty}E_i)_2$, respectively. Writing $v_-=(v_i)_{i=1}^{j-1}$, we use the coefficient Hilbert norms of \Cref{sec:hilbert-norms} to define $|v_-|_2^2=\sum_{i=1}^{j-1}|v_i|_2^2$. By \Cref{lem:packet-estimates},
\begin{equation*}
    |v_-|_2\leq\norm{v_-}\leq D_j|v_-|_2.
\end{equation*}
Choose a Hilbert norm $|\cdot|_+$ on the later projection of $V$ such that $|v_+|_+\leq\norm{v_+}\leq2|v_+|_+$, and define a Hilbert norm on $V$ by
\begin{equation*}
    |v|_H^2=|v_-|_2^2+|v_+|_+^2.
\end{equation*}
Since $D_j\geq2$, this norm satisfies $|v|_H\leq\norm v\leq D_j|v|_H$ for every $v\in V$.

The next theorem compares the lower overlap bound from \Cref{lem:finite-selection}\ref{it:selection-overlap} with the upper bound in \Cref{prop:finite-obstruction}\ref{it:finite-overlap}. The recursive choice of summands forces unconditional basis constants to grow with $j$ in every closed subspace of $Y$ containing $E_j$, thereby ruling out both DPR-lust and an unconditional basis of $Y$.

\begin{theorem}\label{thm:no-unconditional}\label{thm:separated}
Let $Y$ be the space in \eqref{eq:constructed-range}. For every $j\geq1$, if $F$ is a closed subspace with $E_j\subseteq F\subseteq Y$, then every unconditional basis of $F$ has constant greater than $\Lambda_j^{1/8}/5$. In particular,
\begin{equation*}
 \inf\{u(F): E_j\subseteq F\subseteq Y,\ \dim F<\infty\}
 \geq \frac{\Lambda_j^{1/8}}5
 \longrightarrow\infty\qquad(j\to\infty).
\end{equation*}
Consequently, $Y$ has neither DPR-lust nor an unconditional basis.
\end{theorem}

\begin{proof}
Fix $j\geq1$ and such a subspace $F$. The coordinate projection onto $E_j$ leaves $F$ invariant, because $E_j\subseteq F$. Hence
\begin{equation*}
 F=E_j\oplus_2(F\cap W_j)
\end{equation*}
isometrically. Every subspace of $F\cap W_j$ of dimension at most $3q_j$ is $D_j$-Hilbertian by the preceding estimates for the earlier and later summands. Suppose that $F$ has a $C$-unconditional basis with $C\leq\Lambda_j^{1/8}/5$. We apply \Cref{lem:finite-selection} with $a_j=2/\alpha_{p_j}^{n_j}$ and $L_j\leq2$, by \Cref{prop:finite-obstruction}. Since $D_j\leq\Lambda_j^{1/8}$, $\alpha_{p_j}^{n_j}\geq\Lambda_j\geq1$ and
\begin{equation*}
 \eta_j=\frac{2D_jC^2}{\alpha_{p_j}^{n_j}}
 \leq\frac{2}{25}\Lambda_j^{-5/8}<\frac18,
\end{equation*}
the lemma gives a subspace $F'\subseteq E_j\oplus_2H$ of dimension at most $3q_j$, with a $D_jC$-unconditional basis and $\theta_{E_j}(F')\geq1/(16C^2)$. The finite-dimensional overlap estimate therefore yields
\begin{equation*}
 \frac1{16C^2}
 \leq\theta_{E_j}(F')
 \leq\frac{80(D_jC)^3}{\Lambda_j}
 \leq80C^3\Lambda_j^{-5/8}.
\end{equation*}
It follows that
\begin{equation*}
 1\leq1280C^5\Lambda_j^{-5/8}
 \leq\frac{1280}{5^5}<1,
\end{equation*}
a contradiction. The argument applies to every closed subspace $F$ containing $E_j$, since the Hilbert approximation is needed only on the subspace selected by the lemma. Restricting to finite-dimensional $F$ proves failure of DPR-lust, while taking $F=Y$ for all $j\geq1$ rules out an unconditional basis.
\end{proof}

Growth of $u(E_j)$ alone would not exclude DPR-lust, whose definition allows passage to a larger finite-dimensional subspace of $Y$. We next realise $Y$ as the range of a projection with norm arbitrarily close to one on a space with a $1$-unconditional basis.
\subsection{Finite-dimensional embeddings}\label{sec:embedding}

We now embed the tensor spaces into finite-dimensional $\ell_p$ spaces. Taking the $\ell_2$-sum of these embeddings will realise an isomorphic copy of $Y$ as a complemented subspace of $X_\rho$, with a projection of norm less than $1+\rho$. Over $\R$, a uniform circle grid preserves covariance exactly and gives explicit projections, whereas over $\C$, conditional expectations onto finite partitions give the required approximation.

\subsubsection{The real case}
For $\mathbb K=\R$, we sample the functions $f_x$ on a uniform grid in $\T^n$.

For each $n\geq1$, put $M=1000n^2$ and $\vartheta_r=2\pi r/M$ for $r\in\{0,\ldots,M-1\}$. Let $U_n$ be the space of real functions on $\{0,\ldots,M-1\}^n$, with the $L_{p_n}$ norm for normalised counting measure. Define operators $A_n\colon E_n\to U_n$ and $B_n\colon U_n\to E_n$ by
\begin{equation}\label{eq:sampling-maps}
 \begin{split}
 (A_nx)_r&=f_x(\vartheta_{r_1},\ldots,\vartheta_{r_n}),\\
 B_nz&=M^{-n}\sum_{r\in\{0,\ldots,M-1\}^n}
             z_rW(\vartheta_{r_1},\ldots,\vartheta_{r_n}).
 \end{split}
\end{equation}
For $M\geq3$, the grid averages of $\cos(2t)$ and $\sin(2t)$ vanish, so $M^{-1}\sum_{r=0}^{M-1}w(\vartheta_r)w(\vartheta_r)^{\mathsf T}=I_{\R^2}$. Taking tensor products therefore gives
\begin{equation}\label{eq:exact-sampling}
 B_nA_n=I_{E_n}.
\end{equation}

Thus $A_nB_n$ is a projection onto $A_n(E_n)$. The next lemma estimates these maps and shows that the projection norms tend to one.

\begin{lemma}\label{lem:sampling}
With $b_M=1+64\pi/M$ and $p=p_n$, we have
\begin{equation}\label{eq:sampling-bounds}
 \norm{A_n}\leq b_M^n,\qquad
 \norm{B_n}\leq(c_pb_M)^n,\qquad
 \norm{A_nB_n}\leq(c_pb_M^2)^n\longrightarrow1.
\end{equation}
\end{lemma}

\begin{proof}
For $1\leq r\leq4$, the function $|\sqrt2\cos(t-\phi)|^r$ is $16$-Lipschitz, uniformly in $\phi$, so its grid average differs from its integral by at most $32\pi/M$. By \eqref{eq:explicit-alpha} and translation invariance, its integral equals $\alpha_r^r$. This integral is at least $1/2$: for $r\geq2$, monotonicity of the $L_r$ norms on the normalised circle gives $\alpha_r^r\geq\alpha_2^r=1$, while for $1\leq r\leq2$ we have $\alpha_1=2\sqrt2/\pi<1$ and
\begin{equation*}
 \alpha_r^r\geq\alpha_1^r\geq\alpha_1^2=\frac8{\pi^2}>\frac12.
\end{equation*}
The preceding estimates show that every real linear combination of $\cos t$ and $\sin t$ has discrete $L_r$ norm at most $b_M$ times its $L_r(\T)$ norm.

For $x\in E_n$, fixing all but one variable makes $f_x$ a real linear combination of $\cos t$ and $\sin t$ in that variable. Apply the comparison at exponent $p$ successively in each variable, integrating or averaging over the remaining variables at each step. This gives
\begin{equation*}
 \norm{A_nx}_{U_n}\leq b_M^n\norm{f_x}_{L_p(\T^n)}=b_M^n\norm{x}_{E_n}.
\end{equation*}

To estimate $B_n$, first consider $z=(z_r)_{r=0}^{M-1}$ on the one-dimensional grid and put $v=M^{-1}\sum_{r=0}^{M-1}z_rw(\vartheta_r)$. For every $a\in\R^2$ with $|a|_2=1$, we have
\begin{equation*}
 |\langle v,a\rangle|
 \leq\left(M^{-1}\sum_{r=0}^{M-1}|z_r|^p\right)^{1/p}
     \left(M^{-1}\sum_{r=0}^{M-1}|\langle w(\vartheta_r),a\rangle|^{p'}\right)^{1/p'}
 \leq b_M\alpha_{p'}\left(M^{-1}\sum_{r=0}^{M-1}|z_r|^p\right)^{1/p},
\end{equation*}
where the first inequality is H\"older's and the second follows from the grid comparison at exponent $p'$ and $\norm{\langle w(\cdot),a\rangle}_{L_{p'}(\T)}=\alpha_{p'}$. Taking the supremum over $a$ and recalling $c_p=\alpha_p\alpha_{p'}$, we obtain
\begin{equation*}
 \norm{\langle v,w(\cdot)\rangle}_{L_p(\T)}
 =\alpha_p|v|_2
 \leq c_pb_M\left(M^{-1}\sum_{r=0}^{M-1}|z_r|^p\right)^{1/p}.
\end{equation*}

For $z\in U_n$, apply this map in each coordinate, keeping the other coordinates fixed. At each step, integrate the $p$-th power of the one-variable estimate over the remaining coordinates. Fubini's theorem then gives a factor $c_pb_M$ in the product $L_p$ norm at each step. After all $n$ coordinates have been treated, the resulting function is $f_{B_nz}$ by \eqref{eq:sampling-maps}, so
\begin{equation*}
 \norm{B_nz}_{E_n}=\norm{f_{B_nz}}_{L_p(\T^n)}
 \leq(c_pb_M)^n\norm z_{U_n}.
\end{equation*}

Finally, $\norm{A_nB_n}\leq\norm{A_n}\norm{B_n}\leq(c_pb_M^2)^n$. By \eqref{eq:parameter-limits}, $n\log c_{p_n}\to0$, while $M=1000n^2$ gives
\begin{equation*}
 0\leq n\log b_M\leq\frac{64\pi n}{M}=\frac{64\pi}{1000n}\longrightarrow0.
\end{equation*}
The upper bound for the projection norm therefore tends to one.
\end{proof}

\subsubsection{The complex case}
The complex analogue uses the continuous projection $P_n$ from \Cref{sec:tensors}. The next lemma obtains the finite-dimensional embedding by conditional expectations; see \cite[Section~5.1, Theorem~33]{JohnsonSchechtman2001} for similar quantitative results on complemented embeddings into finite-dimensional $\ell_p$ spaces.

\begin{lemma}\label{lem:complex-sampling}
Suppose $\mathbb K=\C$. For each $n\geq1$ and $0<\delta<1$, there are a finite-dimensional complex $\ell_{p_n}$ space $U$ and operators $A\colon E_n\to U$ and $B\colon U\to E_n$ such that
\begin{equation*}
 BA=I_{E_n},\qquad \norm A\leq1,\qquad
 \norm B\leq\frac{c_{p_n}^n}{1-\delta}.
\end{equation*}
\end{lemma}

\begin{proof}
Let $(C_r)_{r=1}^{\infty}$ be conditional expectations onto increasing finite partitions generating the product sphere $\sigma$-algebra. Their strong convergence on $L_{p_n}$ and the finite dimension of $E_n$ imply that $T_r=P_nC_r|_{E_n}\to I_{E_n}$ in operator norm. Choose $r$ so that $\norm{I_{E_n}-T_r}<\delta$ and set
\begin{equation*}
 U=\ran C_r,\qquad A=C_r|_{E_n},\qquad B=T_r^{-1}P_n|_U.
\end{equation*}
Then $BA=I_{E_n}$, and the bounds for $A$ and $B$ follow from $\norm{C_r}=1$, $\norm{P_n}\leq c_{p_n}^n$ and $\norm{T_r^{-1}}\leq(1-\delta)^{-1}$. The norm on $U$ is a weighted $\ell_{p_n}$ norm, so rescaling the coordinates identifies $U$ isometrically with $\ell_{p_n}^{N}$, where $N=\dim U$.
\end{proof}

\subsubsection{The complemented embedding}\label{sec:complemented-embedding}
For a uniformly bounded family of operators $(T_j\colon V_j\to W_j)_{j\in J}$ between Banach spaces, we denote by $\diag(T_j:j\in J)$ the \emph{diagonal operator} from $(\bigoplus_{j\in J}V_j)_2$ to $(\bigoplus_{j\in J}W_j)_2$ given by
\begin{equation*}
 \diag(T_j:j\in J)\bigl((x_j)_{j\in J}\bigr)=(T_jx_j)_{j\in J}.
\end{equation*}
Its norm is $\sup_{j\in J}\norm{T_j}$. We now combine the embeddings over either field.

\begin{proof}[Construction of the projection]
By \Cref{lem:sampling,lem:complex-sampling} and \eqref{eq:parameter-limits}, the sequence $(n_j)_{j=1}^{\infty}$ in \Cref{thm:separated} can be chosen so that there are $U_j=\ell_{p_j}^{N_j}$ and uniformly bounded maps $A_j\colon E_j\to U_j$ and $B_j\colon U_j\to E_j$ satisfying
\begin{equation*}
 B_jA_j=I_{E_j},\qquad \norm{A_jB_j}\leq1+\rho/2
 \qquad(j\geq1).
\end{equation*}
Over $\R$, these are the sampling maps $A_{n_j},B_{n_j}$ after rescaling the coordinates of $U_{n_j}$. Over $\C$, choose errors $(\delta_j)_{j=1}^{\infty}$ tending to zero so that $c_{p_j}^{n_j}/(1-\delta_j)\leq1+\rho/2$. For $X=(\bigoplus_{j=1}^{\infty}U_j)_2$, the diagonal maps
\begin{equation*}
 A=\diag(A_j:j\in\N)\colon Y\to X,
 \qquad B=\diag(B_j:j\in\N)\colon X\to Y
\end{equation*}
are bounded and satisfy $BA=I_Y$. Thus $A$ is an isomorphism onto the closed subspace $Z=A(Y)$, and $P=AB$ is a projection onto $Z$. Its norm is the supremum of the block norms, so $\norm P\leq1+\rho/2<1+\rho$.

Write $P_j=A_jB_j$ for $j\geq1$. In the real case, with $M_j=1000n_j^2$ and $N_j=M_j^{n_j}$, the inner product of two sampling vectors gives the explicit matrix formula
\begin{equation}\label{eq:explicit-projection}
 P_j(r,s)=\left(\frac2{M_j}\right)^{n_j}
 \prod_{t=1}^{n_j}\cos\left(\frac{2\pi(r_t-s_t)}{M_j}\right),
 \qquad P=\diag(P_j:j\in\N),
\end{equation}
where $r,s\in\{0,\ldots,M_j-1\}^{n_j}$. In this case, \eqref{eq:exact-sampling} shows that $P_j^2=P_j$ and $\rank P_j=2^{n_j}$; symmetry makes $P_j$ a Hilbert orthogonal projection. Multiplying the coordinates of $U_{n_j}$ by $M_j^{-n_j/p_j}$ identifies this space isometrically with $\ell_{p_j}^{N_j}$ and leaves the matrix of $P_j$ unchanged.

Over either field, taking the coordinate bases in block order gives a $1$-unconditional basis of $X$. The space $Z$ is isomorphic to $Y$ and therefore fails DPR-lust by \Cref{thm:separated}. The $\ell_2$-sum representation also shows that $X$ and $Z$ are separable and reflexive.
\end{proof}

To complete the proofs of \Cref{thm:real-consequences,thm:main}, we must still establish the GL-lust bounds, superreflexivity and the conclusions for the dual and, over $\R$, both complementary summands. We do so in \Cref{sec:lattices}, where we also obtain the Banach-lattice conclusions.

\section{Local unconditional structure and Banach lattices}
\label{sec:lattices}

We combine the failure of DPR-lust in \Cref{thm:separated} with the complemented embedding of \Cref{sec:embedding}. The first two subsections apply over either scalar field and establish the separation for $Z=P(X)$ and its dual. In the real case, we then modify the construction to obtain the same conclusions for both complementary summands, completing \Cref{thm:real-consequences}.

\subsection{Separation of GL-lust and DPR-lust}

We begin with the standard estimate for GL-lust on complemented subspaces.

\begin{lemma}\label{lem:complemented-gl}
Let $X$ be a Banach space with a $1$-unconditional basis, and let $Z$ be a nonzero Banach space. If bounded linear operators $J\colon Z\to X$ and $R\colon X\to Z$ satisfy $RJ=I_Z$, then
\begin{equation*}
 1\leq\chi_{\mathrm{GL}}(Z)\leq\norm J\norm R.
\end{equation*}
\end{lemma}

\begin{proof}
Since $X$ has a $1$-unconditional basis, approximation by finite coordinate spans gives $\chi_{\mathrm{GL}}(X)=1$. The factorisation argument of Gordon and Lewis \cite[Lemma~3.2]{GordonLewis1974}, applied to $J|_V$ and $R$ for each finite-dimensional subspace $V\subseteq Z$, gives
\begin{equation*}
 \chi_{\mathrm{GL}}(Z)\leq\norm J\norm R\,\chi_{\mathrm{GL}}(X)=\norm J\norm R.
\end{equation*}
The lower bound follows because every inclusion of a nonzero subspace has norm one.
\end{proof}

Combining \Cref{lem:complemented-gl} with \Cref{thm:separated} gives the GL/DPR separation for $Z=P(X)$. We also record the geometric properties that follow from its $\ell_2$-sum structure.

\begin{corollary}\label{cor:range-not-lattice}\label{cor:local-structure}
Let $P$ be the projection constructed in \Cref{sec:embedding}, and set $Z=P(X)$. Then:
\begin{resultparts}
\item\label{it:local-lust}
The local unconditional constants satisfy
\begin{equation*}
 \chi_{\mathrm{GL}}(Z)\leq\norm P<1+\rho,
 \qquad \chi_{\mathrm{DPR}}(Z)=\infty.
\end{equation*}
Consequently, $Z$ admits no unconditional basis. If $\mathbb K=\R$, it is not isomorphic to any Banach lattice.
\item\label{it:local-convexity}
Both $X$ and $Z$ are uniformly convex and hence superreflexive.
\item\label{it:local-fdd}
$Z$ has a contractive unconditional finite-dimensional decomposition and the metric approximation property.
\end{resultparts}
\end{corollary}

\begin{proof}
For \ref{it:local-lust}, observe that $Z$ is isomorphic to the space $Y$ of \Cref{thm:separated}, and therefore fails DPR-lust and admits no unconditional basis. Moreover, \Cref{lem:complemented-gl} gives $\chi_{\mathrm{GL}}(Z)\leq\norm P$. If $\mathbb K=\R$, the nonlattice conclusion follows because every Banach lattice has DPR-lust and this property is invariant under isomorphism.

For \ref{it:local-convexity}, recall that $2<p_j\leq3$ for every $j\geq1$. Clarkson's inequality \cite[Theorem~2]{Clarkson1936} gives the following common lower bound for the moduli of the summands:
\begin{equation*}
 \delta_{\ell_p^N}(\varepsilon)
 \geq1-\bigl(1-(\varepsilon/2)^p\bigr)^{1/p}
 \geq\frac{\varepsilon^3}{24}
 \quad(2\leq p\leq3,\ N\geq1,\ 0<\varepsilon\leq2).
\end{equation*}
An $\ell_2$-sum of spaces with a common modulus of uniform convexity is uniformly convex, so both $X$ and its subspace $Z$ have this property.

To prove \ref{it:local-fdd}, let $X_j=\ell_{p_j}^{N_j}$ and $Z_j=P_jX_j$ for $j\geq1$. Since $P$ acts diagonally, we have the isometric decomposition
\begin{equation*}
 Z=\left(\bigoplus_{j=1}^{\infty}Z_j\right)_2.
\end{equation*}
Its coordinate decomposition is contractive and unconditional, so $Z$ has the metric approximation property in its inherited norm, as observed in \Cref{subsec:decompositions}.
\end{proof}

\subsection{Duality}

We next show that $Y^*$ also fails DPR-lust. The proof uses the local Hilbert estimate for quotients in \Cref{lem:local-hilbert}, which allows us to work with the same sequence of summands as before. We use the canonical isometric identification $Y^*=E_j^*\oplus_2W_j^*$ for $j\geq1$.

\Needspace{7\baselineskip}
\begin{proposition}\label{prop:dual-obstruction}
For the space $Y$ in \Cref{thm:separated},
\begin{equation*}
 \lambda_{\mathrm{DPR}}(Y^*;E_j^*)
 \geq\frac{\Lambda_j^{1/8}}{10}
 \longrightarrow\infty\qquad(j\to\infty).
\end{equation*}
In particular, $Y^*$ fails DPR-lust.
\end{proposition}

\begin{proof}
Let $j\geq1$ and let $F$ be a finite-dimensional subspace satisfying $E_j^*\subseteq F\subseteq Y^*$. Since $F$ contains the entire coordinate summand $E_j^*$, we have $F=E_j^*\oplus_2V$ isometrically, where $V=F\cap W_j^*$. Consequently, $F^*=E_j\oplus_2V^*$. By reflexivity and Hahn--Banach, $V^*$ is isometric to the quotient $W_j/V_\perp$, where $V_\perp=\{w\in W_j:f(w)=0\text{ for every }f\in V\}$.

To apply the selection lemma to $F^*$, we need a local Hilbert estimate for $V^*$. Let $S_j$ be the space obtained from $W_j$ by equipping its earlier summands with their coefficient Hilbert norms and leaving the later summands unchanged. The identity from $S_j$ to $W_j$ induces an isomorphism
\begin{equation*}
 S_j/V_\perp\longrightarrow W_j/V_\perp=V^*
\end{equation*}
with distortion at most $D_j$. The sum of the earlier summands in $S_j$ is Hilbertian, while $p_i-2\leq\varepsilon(3q_j,2)$ for every $i>j$. Combining \Cref{lem:local-hilbert} with this distortion bound shows that every subspace of $V^*$ of dimension at most $3q_j$ is $2D_j$-Hilbertian.

Suppose now that $F$ has a $C$-unconditional basis with $C\leq\Lambda_j^{1/8}/10$. The dual basis of $F^*$ is also $C$-unconditional. We may apply \Cref{lem:finite-selection} to $F^*$ with $a=2/\alpha_{p_j}^{n_j}$, $L\leq2$ and $D=2D_j$, since
\begin{equation*}
 \eta=\frac{4D_jC^2}{\alpha_{p_j}^{n_j}}
 \leq\frac1{25}\Lambda_j^{-5/8}<\frac18.
\end{equation*}
The resulting subspace $F'\subseteq E_j\oplus_2H$ satisfies the lower estimate in \Cref{lem:finite-selection}\ref{it:selection-overlap} and the upper estimate in \Cref{prop:finite-obstruction}\ref{it:finite-overlap}:
\begin{equation*}
 \frac1{16C^2}\leq\theta_{E_j}(F')
 \leq\frac{80(2D_jC)^3}{\Lambda_j}
 \leq640C^3\Lambda_j^{-5/8}.
\end{equation*}
Consequently $1\leq10240C^5\Lambda_j^{-5/8} \leq10240/10^5<1$, a contradiction.
\end{proof}

Together with the duality of GL-lust, \Cref{prop:dual-obstruction} gives the analogue of \Cref{cor:local-structure} for $Z^*$.

\Needspace{14\baselineskip}
\begin{corollary}\label{cor:dual-consequences}
For $Z=P(X)$ as in \Cref{cor:local-structure}, the following hold:
\begin{resultparts}
\item\label{it:dual-lust}
The local unconditional constants of $Z^*$ satisfy
\begin{equation*}
 \chi_{\mathrm{GL}}(Z^*)\leq\norm P<1+\rho,
 \qquad \chi_{\mathrm{DPR}}(Z^*)=\infty.
\end{equation*}
Consequently, $Z^*$ admits no unconditional basis. If $\mathbb K=\R$, it is not isomorphic to any Banach lattice.
\item\label{it:dual-convexity}
$Z^*$ is superreflexive.
\item\label{it:dual-fdd}
$Z^*$ has a contractive unconditional finite-dimensional decomposition and the metric approximation property.
\end{resultparts}
\end{corollary}

\begin{proof}
For \ref{it:dual-lust}, the duality identity for GL-lust \cite[Remark~1.1]{Pisier1978} and \Cref{cor:local-structure}\ref{it:local-lust} give
\begin{equation*}
 \chi_{\mathrm{GL}}(Z^*)=\chi_{\mathrm{GL}}(Z)\leq\norm P.
\end{equation*}

The isomorphism between $Z^*$ and $Y^*$, together with \Cref{prop:dual-obstruction}, shows that $Z^*$ fails DPR-lust. It therefore admits no unconditional basis and, if $\mathbb K=\R$, cannot be isomorphic to a Banach lattice.

\ref{it:dual-convexity} follows from \Cref{cor:local-structure}\ref{it:local-convexity}, since superreflexivity passes to duals. For \ref{it:dual-fdd}, use the isometric block decomposition in the proof of \Cref{cor:local-structure}\ref{it:local-fdd}. Writing $Z_j=P_j(\ell_{p_j}^{N_j})$ for $j\geq1$, we obtain $Z^*=(\bigoplus_{j=1}^{\infty}Z_j^*)_2$ isometrically. This decomposition is contractive and unconditional, and its initial block projections give the metric approximation property.
\end{proof}

\subsection{Two complementary summands without lattice structures}
\label{subsec:two-complementary-ranges}

For the remainder of this section, we work over the real field. To obtain the separation for both complementary summands, we shall alternate the finite projections with their complements. This requires a stronger recursive choice, controlling the Hilbertian distortion of all preceding coordinate spaces. Stern's estimate for projections on $L_p$ spaces \cite{Stern2015} will ensure that the complementary projections also have norms tending to one.

\Needspace{5\baselineskip}
Recall that a class of Banach spaces, closed under isomorphism, is called \emph{primary} if, whenever $X$ belongs to the class and $X=Y\oplus Z$ is a topological decomposition, at least one of $Y$ and $Z$ also belongs to the class. For the class of Banach lattices, membership is understood up to Banach-space isomorphism.

The next corollary gives the GL/DPR separation for both complementary summands and their duals, proving \Cref{cor:separable-nonprimarity}.

\begin{corollary}\label{cor:two-nonlattice-ranges}
For every $\rho>0$ there are a separable superreflexive real Banach lattice $X$ with a $1$-unconditional basis and a projection $Q$ on $X$ such that
\begin{equation*}
 \norm Q<1+\rho,\qquad \norm{I_X-Q}<1+\rho.
\end{equation*}
The spaces $QX$, $(I_X-Q)X$ and their duals satisfy:
\begin{resultparts}
\item\label{it:two-lust}
Each has GL-lust but fails DPR-lust, and none is isomorphic to a Banach lattice.
\item\label{it:two-fdd}
Each has a contractive unconditional finite-dimensional decomposition and the metric approximation property.
\end{resultparts}
Consequently, the class of separable real Banach lattices is not primary.
\end{corollary}

\begin{proof}
For each $n\geq1$, let $R_n=A_nB_n$ be the projection on the sampling space $U_n$ from \Cref{sec:embedding}. We have $\dim U_n=(1000n^2)^n$ and $\rank R_n=2^n$. Thus both $R_n$ and $I_{U_n}-R_n$ are nonzero. Since $U_n=\ell_{p_n}^{(1000n^2)^n}$ with its normalised counting measure, Stern's projection estimate, together with his computation of the Banach--Mazur constant of $L_{p_n}$ spaces \cite[Sections~4 and~5]{Stern2015}, gives
\begin{equation*}
 1\leq\norm{I_{U_n}-R_n}\leq2^{1-2/p_n}\norm{R_n}
 \longrightarrow1,
\end{equation*}
where the convergence follows from $p_n\to2$ and \Cref{lem:sampling}. We may therefore restrict to indices for which $\max\{\norm{R_n},\norm{I_{U_n}-R_n}\}\leq1+\rho/2$ and $\max\{\norm{A_n},\norm{B_n}\}\leq2$.

We now choose $(n_j)_{j=1}^{\infty}$ as in \Cref{sec:infinite-sum}, with one modification. We write $E_j=E_{n_j}$, $V_j=U_{n_j}$ and $N_j=\dim V_j$, and retain the notation $p_j,q_j,\Lambda_j$ for the other selected quantities, for $j\geq1$. The Hilbert bound at the $j$th stage is taken to be
\begin{equation*}
 D_j=\max\Bigl(\{2\}\cup
       \{\alpha_{p_i}^{n_i},N_i^{1/2-1/p_i}:1\leq i<j\}\Bigr),
\end{equation*}
and we retain the conditions
\begin{equation}\label{eq:full-head-separation}
 \Lambda_j\geq D_j^8,\qquad
 p_j-2\leq\varepsilon(3q_i,2)\quad(1\leq i<j).
\end{equation}
Since $D_j$ depends only on the preceding choices, the recursion remains possible. The additional term accounts for the entire space $V_i$, whose Hilbert distortion is at most $N_i^{1/2-1/p_i}$ by the usual comparison of finite-dimensional $\ell_p$ norms.

We claim that, with this choice, every sum of the form
\begin{equation*}
 Z=\left(\bigoplus_{i=1}^{\infty}Z_i\right)_2,\qquad
 Z_i=E_i\ \text{or}\ Z_i\subseteq V_i\quad(i\geq1)
\end{equation*}
fails DPR-lust whenever $J=\{j\in\N:Z_j=E_j\}$ is infinite. To see this, fix $j\in J$ and consider a subspace of dimension at most $3q_j$ of $W_j=(\bigoplus_{\substack{i=1\\i\ne j}}^{\infty}Z_i)_2$. Its projection onto the earlier summands is $D_j$-Hilbertian by the definition of $D_j$, while its projection onto the later summands is $2$-Hilbertian by \Cref{lem:local-hilbert} and \eqref{eq:full-head-separation}. Taking the $\ell_2$-sum of these two Hilbert norms shows that the original subspace is $D_j$-Hilbertian. For every $j\in J$ and every closed subspace $F$ with $E_j\subseteq F\subseteq Z$, we again have $F=E_j\oplus_2(F\cap W_j)$ isometrically. The same local Hilbert bound therefore applies to $F\cap W_j$. The proof of \Cref{thm:separated}, with these constants, gives
\begin{equation*}
 \lambda_{\mathrm{DPR}}(Z;E_j)
 \geq\frac{\Lambda_j^{1/8}}5
 \longrightarrow\infty\qquad(j\in J,\ j\to\infty).
\end{equation*}
Thus every sum of this form fails DPR-lust and has no unconditional basis.

Let $F_j=\ker R_{n_j}$ for $j\geq1$. Define $Q_j\colon V_j\to V_j$ by $Q_j=R_{n_j}$ for odd $j$ and $Q_j=I_{V_j}-R_{n_j}$ for even $j$. The required ambient space and projection are
\begin{equation*}
 X=\left(\bigoplus_{j=1}^{\infty}V_j\right)_2,\qquad
 Q=\diag(Q_j:j\in\N).
\end{equation*}
Then $\norm Q,\norm{I_X-Q}\leq1+\rho/2<1+\rho$. $QX$ is isomorphic to the sum with $E_j$ at the odd indices and $F_j$ at the even indices; for $(I_X-Q)X$ the roles are reversed. Indeed, the operators $A_{n_j}$ on the $E_j$ summands, together with the identities on the $F_j$ summands, define these isomorphisms. Their inverses are given by the restrictions of $B_{n_j}$ and the same identities, so the maps and their inverses have norm at most two. Each of the two sums contains infinitely many of the summands $E_j$, and hence $QX$ and $(I_X-Q)X$ fail DPR-lust.

The quotient argument of \Cref{prop:dual-obstruction} also applies to each of these sums. For each $j\in J$, the sum of the earlier summands has Hilbert distortion at most $D_j$, and the later summands satisfy the same local estimates. Since $J$ is infinite for each sum, both dual spaces fail DPR-lust.

The space $X$ is a separable uniformly convex Banach lattice with its canonical $1$-unconditional basis. \Cref{lem:complemented-gl} and the duality identity used in \Cref{cor:dual-consequences} give, for $R\in\{Q,I_X-Q\}$,
\begin{equation*}
 \chi_{\mathrm{GL}}(RX)\leq\norm R<1+\rho,\qquad
 \chi_{\mathrm{GL}}((RX)^*)\leq\norm R<1+\rho.
\end{equation*}
Since all four spaces fail DPR-lust, none is isomorphic to a Banach lattice. This proves \ref{it:two-lust}.

For \ref{it:two-fdd}, observe that all four spaces are superreflexive and are, isometrically, $\ell_2$-sums of finite-dimensional spaces. Their block decompositions are therefore contractive and unconditional and yield the metric approximation property, as in \Cref{cor:local-structure}\ref{it:local-fdd} and \Cref{cor:dual-consequences}\ref{it:dual-fdd}.
\end{proof}
\section{Geometric properties of the examples}\label{sec:geometry}

We now record some geometric properties of the examples. We show that the spaces and their duals admit conditional asymptotically monotone bases, and conclude with two consequences of their $\ell_2$-sum structure.

\subsection{Conditional bases}\label{sec:schauder}

The spaces $Z_\rho$ of \Cref{thm:main} admit asymptotically monotone bases. In the real construction, the finite sampling projections have Fourier subspaces as both their ranges and their kernels, so we begin with an estimate for the basis constants of such spaces.

\begin{lemma}\label{lem:finite-fourier-bases}
There is an absolute constant $K\geq1$ such that the following holds. Let $n\geq1$, $M\geq3$ and $2\leq p\leq4$, and give $G=(\mathbb Z/M\mathbb Z)^n$ its uniform probability measure. Every real Fourier subspace of $L_p(G)$ whose frequency set is invariant under negation has an $L_2$-orthonormal basis with basis constant at most $K^{2-4/p}$.
\end{lemma}

\begin{proof}
Order the characters
\begin{equation*}
 \chi_\nu(r)=\exp\left(\frac{2\pi i}{M}
                  \sum_{s=1}^n\nu_sr_s\right),
 \qquad \nu,r\in\{0,\ldots,M-1\}^n,
\end{equation*}
lexicographically by frequency. For $0\leq m\leq M^n$, let $K_m$ be the set of the first $m$ frequencies and $F_m$ the corresponding Fourier projection. Reversing the coordinate order identifies these with the Vilenkin partial-sum projections on $(\mathbb Z/M\mathbb Z)^{\N}$ restricted to functions of the first $n$ coordinates. Young's uniform $L_p$ estimate \cite{Young1976}, recalled in \cite[(1.2)]{Young1991}, therefore gives $\norm{F_m}_{4\to4}\leq C_4$ for an absolute constant $C_4\geq1$.

\Needspace{8\baselineskip}
To obtain a real basis, we group frequencies with their negatives, since $\overline{\chi_\nu}=\chi_{-\nu}$. Let $Cf(r)=f(-r)$. This is an isometric involution satisfying $C\chi_\nu=\chi_{-\nu}$, so $CF_mC$ is the Fourier projection onto $-K_m$. Since Fourier projections commute,
\begin{equation*}
 S_m=F_m+CF_mC-F_mCF_mC
\end{equation*}
is the Fourier projection onto $K_m\cup(-K_m)$: the last term subtracts the projection onto the intersection. This frequency set is invariant under negation, so $S_m$ preserves real-valued functions. Moreover, $S_m$ is orthogonal on $L_2$ and has $L_4$ norm at most $2C_4+C_4^2$.

Order the orbits $\{\nu,-\nu\}$ by their first lexicographic occurrence. For a two-element orbit take $\sqrt2\operatorname{Re}\chi_\nu$ followed by $\sqrt2\operatorname{Im}\chi_\nu$; for a one-element orbit take the real character $\chi_\nu$. These vectors form an orthonormal basis of the real space $L_2(G)$. For each two-element orbit, the complex span of these two real vectors is $\operatorname{span}\{\chi_\nu,\chi_{-\nu}\}$. Thus an initial projection ending after a complete orbit is $S_m$, where $m$ is the position at which that orbit first appears.

It remains to consider an initial projection ending after only the real part of a two-element orbit. If that orbit first appears at position $m$, this projection is $S_{m-1}+P_u$, where $P_u$ is the orthogonal projection onto $u=\sqrt2\operatorname{Re}\chi_\nu$. Since $P_uf=(\int_G f\overline u)\,u$ and $|u|\leq\sqrt2$, H\"older's inequality gives
\begin{equation*}
 \norm{P_u}_{4\to4}\leq\norm u_4\norm u_{4/3}\leq2.
\end{equation*}
Consequently, every initial projection has $L_2$ norm at most one and $L_4$ norm at most $K=2C_4+C_4^2+2$. These estimates hold on the full complex scalar spaces, so interpolation between $L_2$ and $L_4$, followed by restriction to real functions, gives the bound $K^{2-4/p}$.

A frequency set invariant under negation selects a subsequence of this real basis. Since the omitted coefficients vanish on the corresponding subspace, its initial projections are restrictions of those of the full basis and satisfy the same bound.
\end{proof}

In \Cref{sec:embedding}, the projection $R_n=A_nB_n$ on $U_n$ has Fourier support $\{-1,1\}^n$, since its matrix kernel, for $r,t\in\{0,\ldots,M-1\}^n$, is
\begin{equation*}
 \left(\frac2M\right)^n
       \prod_{s=1}^n\cos\left(\frac{2\pi(r_s-t_s)}M\right)
 =\frac1{M^n}\sum_{\varepsilon\in\{-1,1\}^n}
       \exp\left(\frac{2\pi i}{M}
                    \sum_{s=1}^n\varepsilon_s(r_s-t_s)\right).
\end{equation*}
Both this set and its complement are invariant under negation. Applying \Cref{lem:finite-fourier-bases} with $M=1000n^2$ and $p_n=2+n^{-3/4}$, we obtain bases of $R_nU_n$ and $\ker R_n$ with constants
\begin{equation}\label{eq:sampled-basis-constants}
 \beta_n=K^{2-4/p_n}\longrightarrow1.
\end{equation}
For these choices of $M$ and $p_n$, bases with constants tending to one can also be obtained from the Lebesgue-constant estimates of Blahota, Persson and Tephnadze \cite[Theorem~1]{BlahotaPerssonTephnadze2018}, using the same pairing argument and interpolation between $L_2$ and $L_\infty$.

An analogous ordering is available for the complex sphere spaces of \Cref{sec:tensors}. For $t=(t_s)_{s=1}^n\in\Omega^n$, write $w_{s,k}(t)=w_k(t_s)$ for $1\leq s\leq n$ and $k\in\{1,2\}$. Consider the orthonormal monomials
\begin{equation*}
 h_\iota=\prod_{s=1}^n\overline{w_{s,\iota_s}},
 \qquad \iota\in\{1,2\}^n.
\end{equation*}
The measure-preserving circle action $w_{s,k}\mapsto e^{-ik3^{s-1}t}w_{s,k}$, for $1\leq s\leq n$ and $k\in\{1,2\}$, assigns these monomials the distinct positive frequencies $\sum_{s=1}^n\iota_s3^{s-1}<3^n$. Ordering them by frequency makes each initial projection a restriction of a spectral interval projection for this action. Such a projection has norm at most one on the full scalar $L_2$ space and at most $1+n\log3$ on $L_\infty$. Indeed, it is given by integration against a Dirichlet kernel, whose normalised $L_1$ norm for an interval of $L$ frequencies is at most $1+\log L$. The latter estimate follows by integrating $\min\{L,\pi/|t|\}$ over $[-\pi,\pi]$. Interpolation therefore gives basis constants $(1+n\log3)^{1-2/p_n}\to1$. Under the finite embeddings, these basis constants increase by a factor of at most $c_{p_j}^{n_j}/(1-\delta_j)$ for $j\geq1$, which tends to one with the choice $(\delta_j)_{j=1}^{\infty}\to0$ made in \Cref{sec:embedding}.

We now pass from the finite-dimensional bases to asymptotically monotone bases of the examples and their duals.

\begin{corollary}\label{cor:schauder-bases}
The following statements hold for $Z_\rho$ of \Cref{thm:main} and its dual over either scalar field. Over $\R$, they also hold for $QX$ and $(I_X-Q)X$ of \Cref{cor:two-nonlattice-ranges} and their duals.
\begin{resultparts}
\item\label{it:schauder-existence}
Each space admits a conditional basis which is shrinking and boundedly complete.
\item\label{it:schauder-tails}
These bases can be chosen to be asymptotically monotone.
\end{resultparts}
\end{corollary}

\begin{proof}
For \ref{it:schauder-existence}, let $Z$ denote $Z_\rho$, $QX$ or $(I_X-Q)X$, and write $Z=(\bigoplus_{j=1}^{\infty}Z_j)_2$ in its inherited norm. Concatenate the finite bases constructed above. An initial projection acts as the identity on preceding blocks, an initial coefficient projection on the current block, and zero on later blocks. Its norm is therefore bounded by the supremum of the finite basis constants. Reflexivity ensures that these bases are shrinking and boundedly complete \cite[Theorem~1.b.5]{LindenstraussTzafriri1977}. The biorthogonal sequences are bases of the duals, which are also reflexive. An unconditional basis of a dual would, by reflexivity, give an unconditional basis of $Z$, contrary to its construction. All these bases are therefore conditional.

For \ref{it:schauder-tails}, an initial projection has norm at most the basis constant of the block in which it ends. Since these constants tend to one and every block is finite-dimensional, the concatenated basis is asymptotically monotone. The same estimate on $(\bigoplus_{j=r}^{\infty}Z_j)_2$ gives basis constants tending to one as $r\to\infty$. The biorthogonal bases have the same initial projection norms, both on the whole space and on the corresponding dual tails, so these conclusions also hold for the duals.
\end{proof}

Given $\varepsilon>0$, starting the recursive choice sufficiently far out and, over $\C$, taking the embedding errors $(\delta_j)_{j=1}^{\infty}$ sufficiently small makes all these basis constants less than $1+\varepsilon$, simultaneously for these spaces and their duals. Here the examples are chosen depending on both $\rho$ and $\varepsilon$.
\Needspace{6\baselineskip}
\begin{remark}\label{rem:geometric-properties}
The $\ell_2$-sum structure of $Z_\rho$, $QX$, $(I_X-Q)X$ and their duals ensures that every infinite-dimensional closed subspace contains a copy of $\ell_2$ complemented in the whole space; see \cite[proof of Proposition~2.1(b) and Proposition~2.4]{OikhbergSpinu2015}.

The real spaces $Z_\rho$, $QX$ and $(I_X-Q)X$ also have type $2$ and cotype $s$ for every $s>2$, by the standard $L_p$ estimates and $p_j\to2$. They fail cotype $2$ by Kwapie\'n's theorem \cite{Kwapien1972}, since none is isomorphic to a Hilbert space.
\end{remark}

\begin{AIstatement}
OpenAI's ChatGPT 5.6 Sol and 6.0 Astra were used during the development of this work.
\end{AIstatement}

\begin{acknowledgements}
The author acknowledges with thanks funding from the EPSRC (grant number EP/W524438/1) that has supported his PhD studies.
\end{acknowledgements}


\begin{thebibliography}{99}
\setlength{\baselineskip}{0.95\baselineskip}

\bibitem{Acuaviva2026}
A.~Acuaviva, \emph{The class of Banach lattices is not primary}, Forum Math. Sigma \textbf{14} (2026), Paper No.~e41, 18~pp. \url{https://doi.org/10.1017/fms.2026.10181}.

\bibitem{AubrunMullerHermes2026}
G.~Aubrun and A.~M\"uller-Hermes, \emph{Asymptotic tensor powers of Banach spaces}, Ann. Inst. Fourier (Grenoble) \textbf{76} (2026), no.~3, 1341--1368. \url{https://doi.org/10.5802/aif.3736}.

\bibitem{BenyaminiFlinnLewis1984}
Y.~Benyamini, P.~Flinn, and D.~R. Lewis, \emph{A space without $1$-unconditional basis which is $1$-complemented in a space with a $1$-unconditional basis}, in \emph{Texas Functional Analysis Seminar 1983--1984 (Austin, Tex.)}, Longhorn Notes, Univ. Texas Press, Austin, TX, 1984, 145--149.

\bibitem{BlahotaPerssonTephnadze2018}
I.~Blahota, L.-E. Persson, and G.~Tephnadze, \emph{Two-sided estimates of the Lebesgue constants with respect to Vilenkin systems and applications}, Glasg. Math. J. \textbf{60} (2018), no.~1, 17--34. \url{https://doi.org/10.1017/S0017089516000549}.

\bibitem{CasazzaKalton1996}
P.~G. Casazza and N.~J. Kalton, \emph{Unconditional bases and unconditional finite-dimensional decompositions in Banach spaces}, Israel J. Math. \textbf{95} (1996), 349--373. \url{https://doi.org/10.1007/BF02761046}.

\bibitem{Clarkson1936}
J.~A. Clarkson, \emph{Uniformly convex spaces}, Trans. Amer. Math. Soc. \textbf{40} (1936), no.~3, 396--414. \url{https://doi.org/10.1090/S0002-9947-1936-1501880-4}.

\bibitem{Clarkson1937}
J.~A. Clarkson, \emph{The von Neumann--Jordan constant for the Lebesgue spaces}, Ann. of Math. (2) \textbf{38} (1937), no.~1, 114--115. \url{https://doi.org/10.2307/1968512}.

\bibitem{DeHeviaEtAl2025}
D.~de Hevia, G.~Mart\'inez-Cervantes, A.~Salguero-Alarc\'on, and P.~Tradacete, \emph{A negative solution to the complemented subspace problem for Banach lattices}, J. Eur. Math. Soc. (JEMS), to appear; arXiv:2310.02196v2. \url{https://arxiv.org/abs/2310.02196v2}.

\bibitem{DeHeviaTradacete2025}
D.~de Hevia and P.~Tradacete, \emph{Complemented subspaces of Banach lattices}, Banach J. Math. Anal. \textbf{19} (2025), no.~4, Paper No.~60, 33~pp. \url{https://doi.org/10.1007/s43037-025-00447-0}.

\bibitem{DubinskyPelczynskiRosenthal1972}
E.~Dubinsky, A.~Pe\l czy\'nski, and H.~P. Rosenthal, \emph{On Banach spaces $X$ for which $\Pi_2(\mathcal L_\infty,X)=B(\mathcal L_\infty,X)$}, Studia Math. \textbf{44} (1972), no.~6, 617--648. \url{https://doi.org/10.4064/sm-44-6-617-648}.

\bibitem{EdelsteinWojtaszczyk1976}
I.~S. Edelstein and P.~Wojtaszczyk, \emph{On projections and unconditional bases in direct sums of Banach spaces}, Studia Math. \textbf{56} (1976), no.~3, 263--276. \url{https://doi.org/10.4064/sm-56-3-263-276}.

\bibitem{Figiel1972}
T.~Figiel, \emph{An example of infinite dimensional reflexive Banach space non-isomorphic to its Cartesian square}, Studia Math. \textbf{42} (1972), no.~3, 295--306. \url{https://doi.org/10.4064/sm-42-3-295-306}.

\bibitem{FigielJohnsonTzafriri1975}
T.~Figiel, W.~B. Johnson, and L.~Tzafriri, \emph{On Banach lattices and spaces having local unconditional structure, with applications to Lorentz function spaces}, J. Approx. Theory \textbf{13} (1975), no.~4, 395--412. \url{https://doi.org/10.1016/0021-9045(75)90023-4}.

\bibitem{GordonLewis1974}
Y.~Gordon and D.~R. Lewis, \emph{Absolutely summing operators and local unconditional structures}, Acta Math. \textbf{133} (1974), 27--48. \url{https://doi.org/10.1007/BF02392140}.

\bibitem{JohnsonLindenstraussSchechtman1980}
W.~B. Johnson, J.~Lindenstrauss, and G.~Schechtman, \emph{On the relation between several notions of unconditional structure}, Israel J. Math. \textbf{37} (1980), 120--129. \url{https://doi.org/10.1007/BF02762873}.

\bibitem{JohnsonSchechtman2001}
W.~B. Johnson and G.~Schechtman, \emph{Finite dimensional subspaces of $L_p$}, in \emph{Handbook of the Geometry of Banach Spaces}, Vol.~1 (W.~B. Johnson and J.~Lindenstrauss, eds.), North-Holland, Amsterdam, 2001, 837--870. \url{https://doi.org/10.1016/S1874-5849(01)80021-8}.

\bibitem{KaltonPeck1979}
N.~J. Kalton and N.~T. Peck, \emph{Twisted sums of sequence spaces and the three space problem}, Trans. Amer. Math. Soc. \textbf{255} (1979), 1--30. \url{https://doi.org/10.1090/S0002-9947-1979-0542869-X}.

\bibitem{KaltonWood1976}
N.~J. Kalton and G.~V. Wood, \emph{Orthonormal systems in Banach spaces and their applications}, Math. Proc. Cambridge Philos. Soc. \textbf{79} (1976), no.~3, 493--510. \url{https://doi.org/10.1017/S0305004100052506}.

\bibitem{King2003}
C.~King, \emph{Maximal $p$-norms of entanglement breaking channels}, Quantum Inf. Comput. \textbf{3} (2003), no.~2, 186--190. \url{https://doi.org/10.26421/QIC3.2-9}.

\bibitem{Kwapien1972}
S.~Kwapie\'n, \emph{Isomorphic characterizations of inner product spaces by orthogonal series with vector valued coefficients}, Studia Math. \textbf{44} (1972), no.~6, 583--595. \url{https://doi.org/10.4064/sm-44-6-583-595}.

\bibitem{Lacey1977}
H.~E. Lacey, \emph{Local unconditional structure in Banach spaces}, in \emph{Banach Spaces of Analytic Functions} (J.~Baker, C.~Cleaver, and J.~Diestel, eds.), Lecture Notes in Math., vol.~604, Springer, Berlin, 1977, 44--56. \url{https://doi.org/10.1007/BFb0069205}.

\bibitem{Lewis1979}
D.~R. Lewis, \emph{Ellipsoids defined by Banach ideal norms}, Mathematika \textbf{26} (1979), no.~1, 18--29. \url{https://doi.org/10.1112/S0025579300009566}.

\bibitem{LindenstraussTzafriri1977}
J.~Lindenstrauss and L.~Tzafriri, \emph{Classical Banach spaces I. Sequence spaces}, Ergebnisse der Mathematik und ihrer Grenzgebiete, vol.~92, Springer-Verlag, Berlin--New York, 1977.

\bibitem{MujicaVieira2010}
J.~Mujica and D.~M. Vieira, \emph{Schauder bases and the bounded approximation property in separable Banach spaces}, Studia Math. \textbf{196} (2010), no.~1, 1--12. \url{https://doi.org/10.4064/sm196-1-1}.

\bibitem{OikhbergSpinu2015}
T.~Oikhberg and E.~Spinu, \emph{Subprojective Banach spaces}, J. Math. Anal. Appl. \textbf{424} (2015), no.~1, 613--635. \url{https://doi.org/10.1016/j.jmaa.2014.11.008}.

\bibitem{Passer2015}
B.~Passer, \emph{An approximate version of the Jordan von Neumann theorem for finite-dimensional real normed spaces}, Linear Multilinear Algebra \textbf{63} (2015), no.~1, 68--77. \url{https://doi.org/10.1080/03081087.2013.844234}.

\bibitem{Pelczynski1960}
A.~Pe\l czy\'nski, \emph{Projections in certain Banach spaces}, Studia Math. \textbf{19} (1960), no.~2, 209--228. \url{https://doi.org/10.4064/sm-19-2-209-228}.

\bibitem{Pisier1978}
G.~Pisier, \emph{Some results on Banach spaces without local unconditional structure}, Compositio Math. \textbf{37} (1978), no.~1, 3--19. \url{https://www.numdam.org/item/CM_1978__37_1_3_0/}.

\bibitem{PlebanekSalguero2023}
G.~Plebanek and A.~Salguero-Alarc\'on, \emph{The complemented subspace problem for $C(K)$-spaces: A counterexample}, Adv. Math. \textbf{426} (2023), Paper No.~109103, 20~pp. \url{https://doi.org/10.1016/j.aim.2023.109103}.

\bibitem{Randrianantoanina2001}
B.~Randrianantoanina, \emph{Norm-one projections in Banach spaces}, Taiwanese J. Math. \textbf{5} (2001), no.~1, 35--95. \url{https://doi.org/10.11650/twjm/1500574888}.

\bibitem{Stern2015}
A.~Stern, \emph{Banach space projections and Petrov--Galerkin estimates}, Numer. Math. \textbf{130} (2015), no.~1, 125--133. \url{https://doi.org/10.1007/s00211-014-0658-5}.

\bibitem{Wojtaszczyk1978}
P.~Wojtaszczyk, \emph{On projections and unconditional bases in direct sums of Banach spaces. II}, Studia Math. \textbf{62} (1978), no.~2, 193--201. \url{https://doi.org/10.4064/sm-62-2-193-201}.

\bibitem{Young1976}
W.-S.~Young, \emph{Mean convergence of generalized Walsh-Fourier series}, Trans. Amer. Math. Soc. \textbf{218} (1976), 311--320. \url{https://doi.org/10.1090/S0002-9947-1976-0394022-8}.

\bibitem{Young1991}
W.-S.~Young, \emph{On an estimate of the partial sums of Vilenkin-Fourier}, Canad. Math. Bull. \textbf{34} (1991), no.~3, 426--432. \url{https://doi.org/10.4153/CMB-1991-069-x}.

\end{thebibliography}
\end{document}